\documentclass[11pt, a4paper]{amsart}
\usepackage{amsmath, amssymb,amsthm}
\usepackage{lettrine}
\usepackage[Sonny]{fncychap}
\usepackage[frenchb]{babel}
\usepackage[all]{xy}
\usepackage{tabularx}
\usepackage[applemac]{inputenc}
\usepackage{graphicx} 
\usepackage[dvipsnames]{xcolor}
\newtheorem{theorem}{Théorème}[section]
 
\newtheorem{prop}[theorem]{Proposition}
\newtheorem{lemma}[theorem]{Lemme}
\newtheorem{cor}[theorem]{Corollaire}

\newtheorem{rem}[theorem]{Remarque}
\newcommand{\R}{\mathbb{R}}
\newcommand{\C}{\mathbb{C}}
\newcommand{\Q}{\mathbb{Q}}

\newcommand{\Z}{\mathbb{Z}}

\newcommand{\cC}{{\mathcal{C}}}
\newcommand{\cE}{{\mathcal{E}}}
\newcommand{\cD}{{\mathcal{D}}}

\newcommand{\go}{\mathfrak{o}}

\newcommand{\cO}{\mathcal{O}}

\newcommand{\cR}{\mathcal{R}}
\newcommand{\cH}{\mathcal{H}}
\newcommand{\cT}{\mathcal{T}}
\newcommand{\cW}{\mathcal{W}}
\newcommand{\gm}{\mathfrak{m}}
\newcommand{\gn}{\mathfrak{n}}
\newcommand{\gp}{\mathfrak{p}}
\newcommand{\gq}{\mathfrak{q}}

\newcommand{\m}{\mathrm{min}}

\DeclareMathOperator{\ad}{ad} 
 
\DeclareMathOperator{\End}{End}
\DeclareMathOperator{\Frob}{Frob}
\DeclareMathOperator{\Ind}{Ind}
\DeclareMathOperator{\Gal}{Gal}
\DeclareMathOperator{\GL}{GL}
\DeclareMathOperator{\rH}{H}

\DeclareMathOperator{\Hom}{Hom}
\DeclareMathOperator{\Spec}{Spec}

\DeclareMathOperator{\Tr}{Tr} 
\DeclareMathOperator{\Ps}{Ps} 
\DeclareMathOperator{\Res}{Res}

\def\psim{{\psi_{_{\heartsuit}}}}
 
\title{LES VARI\'ET\'ES DE HECKE-HILBERT AUX POINTS CLASSIQUES DE POIDS $1$}
\author{\underline{BETINA Adel}}

\address{Universitat Politècnica de Catalunya, Facultad de Matemáticas y Estadística, 08034 Barcelona}
\email{adelbetina@gmail.com}

\begin{document}

\maketitle

\begin{abstract}
On montre que la variété de Hecke associée aux formes de Hilbert sur un corps totalement réel $F$ est  lisse aux points correspondant à certaines séries thêta de poids $1$ et  on donne aussi un critère pour que le morphisme poids soit étale en ces points.  Lorsque les séries thêta sont à multiplication réelle, on construit des formes surconvergentes propres généralisées qui ne sont pas classiques et on exprime leurs coefficients de Fourier à l'aide de logarithmes $p$-adiques de nombres algébriques. Notre approche utilise la théorie des déformations et pseudo-déformations galoisiennes.

\end{abstract}

\section{Introduction}

Les formes modulaires de Hilbert de poids $1$ correspondent dans le programme de Langlands à des représentations impaires de dimension deux des groupes de Galois des corps totalement réels. Deligne-Serre \cite{deligne-serre}, Rogawski-Tunnell \cite{R-T} et Ohta \cite{ohta} ont associé à une forme modulaire de poids $1$ une representation Galoisienne qui a la même fonction $L$ et leur méthode est basée sur des  congruences avec des formes modulaires de poids strictement plus grand que $1$. 

Le but de ce travail est de compter le nombre de familles de Hida qui se spécialisent en une forme classique de Hilbert de poids $1$. Géométriquement, c'est équivalent à la description de la structure locale des variété de Hecke-Hilbert au point correspondant à cette forme de poids $1$. On donne, sous certaines hypothèses, une réponse à cette question en utilisant la théorie des déformations galoisiennes de Mazur et quelques outils cohomologiques de la théorie du corps des classes.

On fixe un nombre premier $p$, un corps de nombres totalement réel $F$ de degré $n$ sur $\Q$, $E$ un corps $p$-adique qui décompose $F$, $\gn$ un idéal de l'anneau des entiers $\go$ de $F$ tel que $\mathrm{N}_{F/\Q} (\gn) \geq 4$ et $I_F$ un ensemble de $n$ plongements complexes de $\bar{\Q}$ dans $\bar{\Q}$ qui prolongent les plongements distincts de $F$ dans $\bar{\Q}$. Soit $\cE$ la variété de Hecke-Hilbert de niveau modéré $\mathfrak{n}$ associée à $F$ introduite par F.Andreatta, A.Iovita et V.Pilloni dans \cite{Pilloni}. Il existe un morphisme localement fini $\kappa: \cE \rightarrow \cW_F$ appelé le morphisme poids, où $\cW_F$ est la variété rigide sur $E$ qui représente les morphismes $\Z_{p}^{\times}\times \Res_{\Q}^F \mathbb{G}_m(\Z_p) \rightarrow \mathbb{G}_{m}$. 

Soit $f$ une forme modulaire de Hilbert cuspidale sur $F$, de poids $1$, propre, de niveau modéré $\gn$ et de pente finie dans le cas où $p$ divise le niveau de $f$. Soit $\rho:G_{F} \rightarrow \GL_ {2} (\bar{\Q}_p)$ la représentation galoisienne associée à $f$ par Rogawski-Tunnell \cite{R-T}. Puisque l'image de $\rho$ est finie, pour tout premier $\gp_i$ de $F$ au dessus de $p$, la restriction de $\rho$ au groupe de décomposition $G_{F_{\gp_i}}$ est la somme de deux caractères $\psi'_i \oplus \psi_i''$ dont l'un est non-ramifié. On dit que $f$ est régulière en $p$ si pour tout premier $\gp_i$ de $F$ au dessus de $p$, la condition $ \psi_i'\ne \psi_i''$ est vérifiée. 

Pour déformer $p$-adiquement $f$, on doit choisir une $p$-stabilisation de $f$ de pente finie (voir \cite[\S1.2]{wiles}). Une $p$-stabilisation de $f$ est ordinaire en $p$ et elle définit un point du lieu quasi-ordinaire de la variété de Hecke $\cE$ et même $f$ définit un point de la courbe de Hecke-Hilbert cuspidale de poids parallèle $\cC_F \hookrightarrow \cE$.

Soit $f$ une $p$-stabilisation d'une série thêta $\theta(\psi)$ de poids $1$ et de niveau modéré $\gn$, où $\psi:G_M \rightarrow \bar{\Q}^{\times}_p $ est un caractère d'ordre fini et $M$ une extension quadratique de $F$. La forme modulaire $f$ définit un point $x \in \cE$ (même $x \in \cC_F$) et on note respectivement $\mathcal {T}$, $\cT^{ord}$ les complétés des anneaux locaux de $\cE$ et $\cC_F$ en $x$ et $\Lambda$ le complété de l'anneau local de $\cW_F$ en $\kappa(x)$. 

Soient $\gm_{\Lambda}$ l'idéal maximal de $\Lambda$ et $\cT ' = \cT/\gm _ {\Lambda} \cT$ l'anneau local de la fibre de $\kappa(x)$ en $x$. Notons que $\cT ' = \cT/\gm _ {\Lambda} \cT$ est un anneau artinien, puisque $\kappa$ est localement fini.

Soit $S_p$ (resp. $S^p$) l'ensemble des premiers de $F$ au-dessus de $p$ qui se décomposent dans $M$ (resp. qui  sont inertes ou qui se ramifient dans $M$). Pour tout premier $\gp_i$ de $F$ au dessus de $p$ on note $e_i$ l'indice de ramification de $\gp_i$ et $f_i$ le degré d'inertie de $\gp_i$ (on a $\sum _ {\gp_i \mid p} e_i f_i=[F:\Q]=n$).

\begin{theorem} \label{main-thm-RM}\ Supposons que $f$ est $p$-régulière, $M$ est totalement réel et que la conjecture de Leopoldt est vraie pour $M$, alors : 
\begin{enumerate}

\item la variété $\cE$ est toujours lisse au point $x$ et la dimension de l'espace tangent de la fibre $\cT ' $ de $\kappa(x)$ en $x$ est égale à $\sum_{\gp_i \in S_p} e_i. f_i$.

\item la dimension de l'espace tangent de $\cT ^ {ord}$ est égale à $\max \{1, \sum_{\substack{\gp_i \in S_p}} f_i . e_i\}$.
\end{enumerate}
\end{theorem}

\medskip 

Lorsque $F=\Q$, le théorème ci-dessus a été démontré par Bellaïche et Dimitrov \cite{D-B} et aussi par Cho-Vatsal \cite{cho-vatsal} sous certaines hypothèses supplémentaires.

\medskip 
Soit $S^{\dag}_{1} (\gn,\chi)[[f]]$ l'espace propre généralisé de $f$ dans l'espace des formes modulaires surconvergentes de poids $1$. Sous les les hypothèses du Théorème \ref{main-thm-RM} et si $S_{p}$ est non vide, il existe un morphisme surjectif $\pi:\cT '\twoheadrightarrow \bar{\Q}_p [\epsilon]$ de noyau $I_\pi$. Ainsi, $I_\pi$ annule un sous-espace $S^{\dag}_{1}(\gn,\chi) [I_\pi]$ de $S^{\dag}_{1} (\gn,\chi) [[f]]$ de dimension $2$, et ce sous-espace contient une forme non-classique.

L'évaluation $f_{\mathfrak{c}}$ de toute forme modulaire surconvergente $f$ en les objets de Tate $(\mathbb{G}_m \otimes_{\mathfrak{o}} \mathfrak{c}^{-1} \mathfrak{d}^{-1}_{F})/q$ associés aux pointes standards $\infty(\mathfrak{c};\mathfrak{o})$, où $\mathfrak{c}$ parcourt les éléments de $\mathrm{CL}_F^{+}$ détermine son $q$-développement (voir \cite[\S2.3]{dim-weise}). Après une modification des coefficients de son $q$-développement, on obtient son développement adélique et ses coefficients de Fourier notés $a_{\mathfrak{q}}(f)$ où $\mathfrak{q}$ est un idéal de $\go$.

Suivant les définitions de \cite{darmon}, une forme surconvergente $f^{\dag}$ dans $S^{\dag}_{1} (\gn,\chi) [I_\pi]$ qui n'est pas un multiple de $f$ est appelée une forme généralisée attachée à $f$, et on dit qu'elle est normalisée si son premier coefficient de Fourier $a_{\go}(f^{\dag})$ de la $q$-expansion de $f^{\dag}$ est nul.

Pour n'importe quel idéal premier $\mathfrak{q}$ de $F$, on notera respectivement $a_{\mathfrak{q}}(f^{\dag}) $ et $a_{\mathfrak{q}}(f)$ les coefficients de Fourier de $f^\dag$ et de $f$. Fixons un plongement $\iota_p: \bar{\Q} \hookrightarrow \bar{\Q}_p$ et notons $\log_p$ la détermination standard du logarithme $p$-adique sur $ \bar{\Q}_p^ \times$. Soient $\sigma \in G_F$ un automorphisme non trivial sur $M$, $H$ le corps de nombres fixé par $\ker \psi/\psi^{\sigma}$, $\ell \nmid \gn p$ un premier de $F$ inerte dans $M$ et $\lambda$ un premier de $H$ au dessus de $\ell$. Choisissons  $u_\lambda$ dans $\cO_H[1/\lambda]^{\times} \otimes \Q$ une $\lambda$-unité de $H$ de $\lambda$-valuation égale à $1$. 

Soient $\sigma_\lambda \in \Gal(H_\lambda/F_\ell)\subset \Gal (H/F)$ le Frobenius en $\ell$ attaché à la place finie $\lambda$ et $I'_F$ le sous-ensemble de $I_F$ constitué par les plongements de $\bar{\Q}$ qui induisent les places de $H$ au dessus de $p$ apparaissant dans $S_p$ suivant le diagramme (\ref{place H}) de \S \ref{local unit p-adic}.

Le théorème suivant est une généralisation du théorème \cite[1.1]{darmon}.

\begin{theorem}\label{q exp}
On suppose que: 
\begin{enumerate}

\item $M$ est un corps totalement réel et la conjecture de Leopoldt est vérifiée pour $M$. 
\item l'ensemble $S_p$ n'est pas vide et $p$ est relativement premier au niveau de $\theta(\psi)$. 
\item $\theta(\psi)$ est $p$-régulière. 

\end{enumerate}

Soit $f$ une $p$-stabilisation de $\theta(\psi)$ telle que $\psi^{\sigma}(\Frob_{\gp_i})$ est la valeur propre de l'opérateur $U_{\gp_i}$ quand $\gp_i \in S_p$. Alors, on peut associer linéairement à tout $(\alpha_i)_{i \in I'_F} \in \bar{\Q}_p^{|I'_F|}$ un morphisme surjectif $\pi: \cT' \twoheadrightarrow \bar{\Q}_p[\epsilon]$ tel que pour tout $\ell\nmid \gn p$, {\small 
 $$ a_\ell(f^\dag) = \left\{ 
 \begin{array}{cl}
 0 & \mbox{ si  } \ell \text{ se décompose dans } M; \\
 \psi(\sigma \sigma_{\lambda})\sum_{g'_i \in I'_F}  \alpha_i \sum_{h \in \Gal(H/M)} \psi/\psi^{\sigma}(h).\log_p(g_i\circ h(u_\lambda))    & \mbox{ si  } \ell \text{ est inerte dans } M.   
 \end{array}
 \right. $$ }
\end{theorem}

\medskip 
Les théorèmes suivants décrivent la structure locale de $\cE$ au point $x$ lorsque $M$ n'est pas totalement réel.

\begin{theorem}\label{CM-case} Soient $K$ un corps quadratique imaginaire dans lequel $p$ est décomposé, $\sigma \in G_\Q$ non trivial sur $K$ et $\xi:G_K \rightarrow \bar{\Q}_p^{\times}$ un caractère d'ordre fini tel que  $\xi/\xi^{\sigma}$ est d'ordre paire. Notons $M$ l'extension biquadratique de $\Q$ contenue dans le corps de nombres fixé par $\ker(\xi/\xi^{\sigma})$ et $F$ le sous-corps quadratique réel de $M$. Supposons que $\psi=\xi_{|G_M}$, $\psi_{|G_{M_{v_i}}} \ne \psi^{\sigma}_{|G_{M_{v_i}}}$ pour toute place première $v_i$ de $M$ au dessus de $p$ et que $(\psi/\psi^{\sigma})^2$ est non trivial. 

Alors le morphisme $\kappa:\cE \rightarrow \cW_F$ est étale au point associé à une $p$-stabilisation $f$ de $\theta(\xi_{|G_M})$.

\end{theorem}

\begin{theorem}\label{thm M mixt}

Soient $p\geq 3$, $\mathbb{F}$ un corps fini de caractéristique $p$ et $\psi$ le relèvement de Teichmüller d'un caractère $\overline{\psi}:G_M\rightarrow \mathbb{F}^{\times}$ tel que $\overline{\psi}/\overline{\psi}^{\sigma}$ n'est pas quadratique. Supposons que :
\begin{enumerate}
\item Tout premier $\mathfrak{q}$ de $F$ qui divise le conducteur de $\bar{\psi}$ se décompose dans $M$, $\bar{\psi}$ est ramifié sur un facteur de $\mathfrak{q}$ et non ramifié sur l'autre et $p$ est non ramifié dans $M$.

\item La restriction de $\bar{\rho}=\Ind^F_M \bar{\psi}$ à $\Gal(\bar{\Q}/{F(\sqrt{(-1)^{(p-1)/2}p)}})$ est absolument irréductible et $\bar{\rho}$ est $p$-distinguée. 
\item $M$ a $2r$ plongements complexes et $[F:\Q]-r-|S^p|-\sum_{\substack{\gp_i \in S^p}} f_i . e_i>0$
\end{enumerate}

Alors la variété rigide $\cE$ est ramifiée sur l'espace des poids $\cW_F$ au point $x$.

\end{theorem}

\subsection*{Remarque:}
Nous avons appris récemment que S.V.Deo a annoncé un résultat similaire au théorème \ref{main-thm-RM} dans \cite{V.Deo} en utilisant une méthode différente pour calculer les dimensions des espaces tangents de nos problèmes de déformation. Il a annoncé un résultat similaire au théorème \ref{thm M mixt} sous l'hypothèse que la conjecture de {\it Schanuel} est vraie pour le corps fixé par $\ker \ad \rho $.

\medskip

Notre étude de la structure locale de la variété rigide $\cE$ au voisinage d'une forme classique de poids $1$ donne l'application suivante sur la théorie des fonctions $L$ $p$-adiques. Mazur, Kitagawa, Panchishkin, Teiltelbaum et Dimitrov ont construit une fonction $L$ $p$-adique pour une famille de Hida  sous certaines hypothèses (voir les conditions (A) ou (B) dans \cite[p.105]{kitagawa}). Leur propriété principale est la suivante : chaque spécialisation en une forme classique non critique $g$ est égale, à une unité près, à la fonction $L$ $p$-adique $L_p(g,s)$ attachée a $g$ (voir \cite{A-V} , \cite{MTT}, \cite{Pan} et \cite{dim L}).
Soit $\cW^{*}_F$ l'espace analytique rigide définie sur $\Q_p$ dont les $A$-points sont en bijection avec $\Hom(G_{F_{p,\infty}}, A^{\times})$, où $G_{F_{p,\infty}}$ est l'extension maximale non ramifiée en dehors de $p$ et des places à l'infini. Ainsi, $\cW^{*}_F= \coprod_ {\epsilon \in \pi_0((F\otimes \mathbb{R})^{\times})} \cW_{F}^{\epsilon}$ où $\cW^{\epsilon}_F$ est l'espace analytique rigide définie sur $\Q_p$ dont les $A$-points sont en bijection avec $\Hom(G_{F_{p}}, A^{\times})$. 

Les méthodes utilisées dans \cite{Bcrit} (détaillées dans \cite[Corollary VI.4.2]{Bbook}) et dans \cite{hansen} donnent le corollaire suivant : 

\begin{cor}\label{L}
Soit $f$ une forme classique de poids $1$ régulière en $p$ qui satisfait soit les hypothèses du théorème \ref{CM-case} ou soit $S_p$ est vide et $M$ est un corps totalement réel qui vérifie la conjecture de Leopoldt. Alors, il existe un voisinage admissible $\mathcal{U}$ de $f$ dans $\cE$, deux constantes réelles $0<c<C$, et une fonction analytique $L_p^\epsilon$ sur $\mathcal{U} \times \cW^\epsilon_F$, telle que pour toute forme propre non critique $g \in \mathcal{U}$,
 et pour tout $s \in \cW^\epsilon_F$  on a 
 $$L_p^\epsilon(g,s) = e^\epsilon(g) L_p(g,s),$$
 où $e^\epsilon(g)$ est une période $p$-adique qui vérifie $c < |e^{\epsilon}(g)|_p < C$.
Pour tout $g\in \mathcal{U}$ les fonctions $s \mapsto L_p^{\epsilon}(g,s)$ sont bornées.
\end{cor}

La fonction $L_p^\epsilon(f_\alpha, .) $ est uniquement déterminée à une constante inversible près, et nous donne aussi une définition naturelle d'une fonction $L$ $p$-adique associée à $f$.

\medskip 

On va expliquer maintenant les principales idées derrière la preuve des théorèmes \S \ref {main-thm-RM} et \S \ref {CM-case}. Dans le \S\ref{Tangent space}, nous introduisons deux problèmes de déformation de $ \rho $ représentables par les anneaux locaux $ \cR$ et $ \cR^{ord} $ tels que $ \cR $ s'envoie surjectivement sur $ \cT $ par la proposition \ref {prop-surj}. Le calcul de l'espace tangent de $ \cR $ représente une partie importante de la preuve. Nous montrons dans les théorèmes \ref {dim c} et \ref {dimension} que la dimension de l'espace tangent de $ \cR $ est toujours $ [F:\Q] + 1 $ et donc la surjection ci-dessus est un isomorphisme d'anneaux locaux réguliers et on a $ \cR ^ {ord} \simeq \cT ^ {ord} $. L'étude de l'algèbre $ \cT '$ donne un critère précis pour que $\kappa$ soit étale en $x$. L'espace tangent de $ \cR $ est isomorphe à un sous-espace de $ \rH ^ 1 (F, \ad \rho) $ qui satisfait certaines conditions locales pour les premiers de $F$ au dessus de $ p $, mais vu que $ \rho $ est d'image finie, par inflation-restriction, l'espace tangent peut être  vu comme un sous-espace de $ (\Hom (G_H, \bar \Q_p) \otimes \ad \rho) ^ {\Gal (H / F)} $, où $H$ est une extension galoisienne finie de $F$. Pour calculer la dimension de cet espace tangent, il suffit de déterminer la dimension du $ \psi/\psi^{\sigma}$-espace propre à l'intérieur de $ \Hom (G_H, \bar {\Q} _p) $. En effet, sa dimension peut être facilement bornée par la théorie du corps de classe, et quand $ M $ est un corps totalement réel, on constate le phénomène suivant: le corps $ H $ est CM et le $ \psi/\psi^{\sigma}$-espace propre à l'intérieur des unités globales est trivial (voir \ref{factoring global}). Ainsi, nous pouvons construire une base de l'espace tangent de $ \cT '$ en termes de logarithme $ p $-adique.
D'autre part, lorsque $ M $ n'est ni totalement réel, ni totalement complexe, on montre que $ \cR_ {\m} \simeq \cT $ en utilisant un résultat de Fujiwara \cite{Fujiwara}.

\medskip

\section{Vari\'et\'es de Hecke-Hilbert }\label{wt_1h}

Les formes modulaires $p$-adiques de Hilbert sur $F$ peuvent être vues comme des fonctions sur les $\Z_p^{\times} \times \mathrm{Res}_{\Q}^F \mathbb{G}_m(\Z_p)$-torseurs de la trivialisation de la tour d'Igusa qui est le pro-revêtement étale du lieu ordinaire de l'analytifié du schéma modulaire de Hilbert donné par la limite projective des duaux des sous-groupes canoniques. Les fonctions homogènes de poids $(w,\upsilon)$ pour l'action de $\Z_p^{\times} \times \mathrm{Res}_{\Q}^F \mathbb{G}_m(\Z_p)$ sont les formes modulaires $p$-adiques de Hilbert de poids $(w,\upsilon)$.  Donc, on obtient un espace de dimension infinie qui se fibre sur l'espace des poids, et grâce à l'existence du sous-groupe canonique sur le lieu ordinaire, on peut plonger l'espace des formes modulaires de Hilbert dans l'espace des formes modulaires $p$-adiques de Hilbert. Pour étudier les familles $p$-adiques de pente finie, V.Pilloni, F.Andreatta et A.Iouvita ont prolongé la construction précédente à un voisinage surconvergent strict du lieu ordinaire dans le but d'appliquer la théorie spectrale de Coleman-Mazur à l'opérateur complètement continue $U_p$ et en utilisant la machinerie des variétés rigides de Hecke due à Buzzard \cite {buzzard}, ils ont obtenu la variété rigide de Hecke-Hilbert désirée $ \cE $ sur $E$. D'autre part, Kisin et Lai ont construits dans \cite {K-L} la courbe de Hecke-Hilbert cuspidale $\cC_F$ de poids parallèle de niveau modéré $ \gn $ en étendant la construction de Coleman-Mazur de la courbe de Hecke \cite {coleman-mazur}. On peut identifier $ \cC_F $ avec le sous-espace fermé de $ \cE $ donné par l'équation $ \upsilon = 0 $. Selon les résultats de \cite {K-L}, il existe un morphisme localement fini surjectif et plat $ \kappa: \cC_F \rightarrow \cW _ {F,\upsilon = 0} $, où $ \cW _ {F,\upsilon = 0} $ est la sous-variété fermé de $\cW _F$ définie par l'équation $ \upsilon = 0 $ ($ \kappa $ est la restriction de $ w $ à $ \cC_F $). Notons que $\cT'$ est aussi l'anneau local en $x$ de la fibre $\kappa^{-1}(\kappa(x))$. Le lieu $ \cE^{n.ord} $ de $ \cE $ où $ | U_p |_p = 1 $ est ouvert et fermé dans $ \cE $ et est appelé le lieu quasi-ordinaire. Il est connu que le lieu quasi-ordinaire $\cE^{n.ord}$ est isomorphe à la fibre générique de l'algèbre de Hecke quasi-ordinaire en $p$.

D'après \cite {buzzard} et \cite {Pilloni}, la variété de Hecke $ \cE $ est équidimensionnelle de dimension $ [F:\Q]+ 1 $, réduite et équipée d'un morphisme localement fini surjectif $ \kappa: \cE \rightarrow \cW_F $ appelé le morphisme poids, où $ \mathcal {W}_F$ est l'espace rigide sur $E$ représentant les morphismes $ \Z_ {p} ^ {\times} \times \Res _ {\Q}^F \mathbb {G} _m (\Z_p) \rightarrow \mathbb{G} _ {m} $. Par construction de $ \cE $, il existe un morphisme $  \Z[(T_\ell)_{\ell \nmid p \mathfrak{n}},U_p] \rightarrow \mathcal {O} _ {\cE} ^ {rig} (\cE) $ tel que les images de $(T_\ell)_{\ell \nmid p \mathfrak{n}}$ et $U_p$ sont des sections globales de $ \mathcal {O} ^ {rig} _ {\cE} $ bornées par $ 1 $; l'application canonique '' système de valeurs propres '' $ \cE (\C_p) \rightarrow \Hom ( \Z[(T_\ell)_{\ell \nmid p \mathfrak{n}},U_p], \C_p) $ est injective et induit une correspondance bijective entre l'ensemble de $ \C_ {p} $-points de $ \cE $ de poids $(w, \upsilon)$ et l'ensemble des formes modulaires cuspidales surconvergentes de Hilbert propres de niveau modéré $\gn$, de poids $(w, \upsilon)$, avec des coefficients de Fourier dans $ \C_ {p} $ et de pente finie. Puisque l'image de $ \Z[(T_\ell)_{\ell \nmid p \mathfrak{n}},U_p]$ dans $\cO^{rig}_{\cE}(\cE)$ est relativement compacte et que la $E$-algèbre $ \mathcal {O} _ {\cE} ^ {rig} (\cE) $ est réduite, il existe un pseudo-caractère continue:

 \begin{equation}\label{pseudo-character }
 \Ps_{\cE}: G_{F,\mathfrak{n} p} \rightarrow \cO(\cE)
 \end{equation}
 qui envoie  $\Frob_\ell$ vers $T_\ell$ pour tout $\ell \nmid \mathfrak{n} p$ (voir \cite[\S5]{Pilloni} and \cite{bellaiche-chenevier-book} pour plus de détails).

D'après le théorème \cite[1.1]{P-S}, le morphisme poids $\kappa:\cE \rightarrow \mathcal{W}_F$ est étale en les points {\it non critiques} et { \it $p$-réguliers}.

En général, on n'a pas beaucoup de résultats sur la géométrie des variétés de Hecke-Hilbert. Par exemple, on ne sait pas si elles ont un nombre fini de composantes ou si elles sont propres sur l'espace des poids (voir \cite{Shin}).  Lorsque $F=\Q$, Diao et Liu ont montré dans \cite{diao} que la courbe de Coleman-Mazur est propre sur l'espace des poids. On sait aussi que la courbe de Coleman-Mazur est lisse en la plupart des points classiques (voir \cite{B-C}, \cite{D-B}, \cite{cho-vatsal}, \cite{coleman-mazur} et \cite{kisin}) bien qu'il existe des points associés à des formes propres classiques et irrégulières en $p$ pour lesquels la courbe de Hecke n'est pas lisse \cite{dimitrov-ghate}.

\subsection*{Séries thêta de poids $1$}

On note respectivement $\go$ et $\cO$ les anneaux des entiers de $F$ et de $E$. Soient $M$ une extension quadratique de $F$, $\sigma \in \Gal(\bar{\Q}/F)$ non trivial sur $M$ et $\mathbb{F}$ le corps résiduel de $\cO$. 

On fixe deux plongements $E\hookrightarrow \bar{\Q}_p$, $\iota_p : \bar{\Q} \hookrightarrow \bar{\Q}_p$ et une clôture algébrique $\bar{\Q} \subset  \C$ de $\Q$. On note $c$ la conjugaison complexe du groupe de Galois absolu $G_{\Q}$ et $\varepsilon_M$ le caractère non trivial de l'extension quadratique $M/F$ et $\Delta$ le groupe de Galois $\Gal(M/F)$.

Soient $\psi: G_M \rightarrow \cO^{\times}$ un caractère d'ordre fini, $\psi^{\sigma}$ le caractère $G_M$ définie par $\psi^{\sigma}(g)=\psi(\sigma^{-1}g\sigma)$, $\psim$  le caractère $\psi/\psi^{\sigma}$ et $\mathfrak{b}$ le conducteur de $\psi$. Ainsi, on peut voir $\psi$ comme un caractère sur le groupe des idèles $\psi: M^\times \backslash M^\times_{\mathbb{A}} \rightarrow \cO^\times$ (où $M^\times_{\mathbb{A}}$ est le groupe des idèles de $M$).

On suppose que pour toute place $\tau$ de $F$ qui reste réelle sur $M$, $\psi$ est trivial sur l'une des places de $M$ au dessus de $\tau$ et non trivial sur l'autre. Soit $\gn$ la partie de $\mathrm{N}_{M/F}(\mathfrak{b}).\Delta_{M/F}$ qui est première à $p$. D'après un théorème de Weil, il existe une série de thêta $\theta(\psi)$ de niveau modéré $\mathfrak{n}$ telle que $L(s,\theta(\psi))=L(s,\psi)$ (i.e $\rho=\Ind^{F}_M\psi$ est la représentation $p$-adique associée à $\theta(\psi)$). Les systèmes de valeurs propres d'une $p$-stabilisation $ f $ de $\theta(\psi)$ correspondent à un point $ x \in \cE (\bar \Q_p) $, de plus $ x $ se trouve sur lieu quasi-ordinaire $ \cE^{n.ord} $ de $ \cE $.

\section{D\'eformations galoisiennes}

On introduit dans cette section trois problèmes de déformations de $\rho$ qu'on notera $\cD$, $\cD^{ord}$ et $\cD^{ord}_{\det \rho}$ et on déterminera leurs espaces tangents respectifs dans la section \ref{M is Real}.
\subsection{Probl\`emes de d\'eformations}\label{Tangent space}\

L'image projective de $\rho$ est diédrale et contient un élément d'ordre $2$ qu'on notera $\sigma$. Soit $(e_1, e_2)$ une base de $\bar{\Q}_p^2$ dans laquelle $\rho_{|G_M} =\psi \oplus {\psi}^\sigma$. Après une renormalisation de $(e_1,e_2)$, on peut supposer que $\rho(\sigma)=\left(
\begin{smallmatrix}
0&1\\
1&0\end{smallmatrix}\right)$ dans $\mathrm{PGL}_{2}(\overline{\Q})$.

Le choix du groupe de décomposition en chaque place $\gp_i$ de $F$ au dessus de $p$ induit une place première canonique $v_i$ de $M$ parmi les places premières de $M$ au dessus de $\gp_i$ et nous permet de voir $G_{M_{v_i}} \subset G_{F_{\gp_i}}$ comme un sous-groupe de décomposition de $G_M$ en $v_i$. Puisque $f$ est $p$-régulière et $\rho$ d'image finie, $\rho$ est {\it ordinaire} en chaque premier $\gp_i$ de $F$ au dessus $p$ dans le sens où la restriction de $\rho$ à  $G_{F_{\gp_i}}$ est l'extension d'un caractère non ramifié $\psi_{i}''$ par un caractère $\psi_{i}'$ tel que $\psi_{i}''(\Frob_{\gp_i})=\alpha_i$, où $\alpha_i$ est la valeur propre de $f$ pour $U_{\gp_i}$. Maintenant, on choisit une base $\{e'_{i,1},e'_{i,2}\}$ de $\bar\Q_p^2$ dans laquelle $\rho_{|G_{F_{\gp_i}}}=\psi_{i}' \oplus  \psi_{i}''$. Si les deux caractères $\psi_{i}'$ et $\psi_{i}''$ sont non ramifiés, on privilégiera $\psi_{i}''$, de sorte que la base $\{e'_{i,1},e'_{i,2}\}$ soit unique à un scalaire près. Dans le cas où $\gp_i$ se décompose dans $M$, on a $(e'_{i,1},e'_{i,2})=(e_1,e_2)$ où $\psi^{\sigma}(\Frob_{\gp_i})=\alpha_{\gp_i}$, sinon $\psi(\Frob_{\gp_i})=\alpha_{\gp_i}$ et $(e'_{i,2},e'_{i,1})=(e_1,e_2)$.

Soit $\mathfrak{C}$ la catégorie dont les objets sont les $\bar{\Q}_p$-algèbre noethériennes, locales, complètes par rapport à l'idéal maximal, de corps résiduel isomorphe à $\bar{\Q}_p$ et dont les morphismes sont les homomorphismes de $\bar{\Q}_p$-algèbres locales. Soient $A$ un anneau dans la catégorie $\mathfrak{C}$ et $\rho_A :G_{F}\rightarrow \GL_{2}(A)$ une déformation de $\rho$, on dit que $\rho_A$ est quasi-ordinaire (resp. ordinaire) en $p$ si, et seulement si, pour tout $\gp_i \mid p$, on a $ (\rho_A)_{|G_{F_{\gp_i}}} \simeq \left(
\begin{smallmatrix}
\psi'_{i,A}&*\\
0 &\psi''_{i,A}\end{smallmatrix}\right) $, où $\psi''_{i,A}$ est un caractère (resp. un caractère non ramifié) qui relève $\psi_{i}''$. On considère le foncteur de déformation $\mathcal{D}^{ord}: \mathfrak{C} \rightarrow \mathrm{SETS}$ (resp. $\cD$), donné par les classes d'équivalences strictes des déformations de $\rho$ qui sont ordinaires en $p$ (resp. quasi-ordinaire en $p$). Puisque la représentation $\rho=\Ind^F_M \psi$ est absolument irréductible, $p$-ordinaire et $p$-regulière, les critères de Schlesinger impliquent que $\mathcal{D}^{ord}$ (resp. $\cD$) est représentable par le $2$-uplet $(\cR^{ord},\rho_{\cR^{ord}})$ (resp. $(\cR,\rho_{\cR})$), où $\rho_{\cR^{ord}}$ (resp. $\rho_{\cR}$) est l'anneau de déformation universel $p$-ordinaire  (resp. quasi-ordinaire en $p$) de $\rho$.  Soit $\cD^{ord}_{\det \rho}$ le sous-foncteur de $\cD^{ord} $ qui consiste en les deformations de determinant fixé.

\subsection{Espaces tangents}\

On note $t_{\cD^{ord}}$ (resp. $t_{\cD}$) l'espace tangent de $\cD^{ord}$ (resp. $\cD$) et $t^{0}_{\cD^{ord}}$ l'espace tangent de $\cD^{ord}_{\det \rho}$.

Il existe une décomposition $$(\ad \rho)_{|G_{F_{\gp_i}}} =  \psi_{i}'/ \psi_{i}' \oplus \psi_{i}'/ \psi_{i}'' \oplus  \psi_{i}'' /  \psi_{i}' \oplus  \psi_{i}''/ \psi_{i}''$$

Le choix de la base $(e'_{i,1},e'_{i,2})$ de $M_{\bar{\Q}_p}$ identifie $\End_{\bar{\Q}_p}(M_{\bar{\Q}_p})$ avec $M_2(\bar{\Q}_p)$; puisque $\rho$ est $p$-ordinaire pour tout premier $\gp_i$ de $F$, on a une application naturelle $G_{F_{\gp_i}}$-équivariante:

\begin{equation} \label{locp}
\begin{split}
  \ad \rho & \overset{C'_{i,*},D'_{i,*} }{\longrightarrow} \psi_{i}''/\psi_{i}' \oplus \psi_{i}''/\psi_{i}'' \\
 \left(\begin{smallmatrix}a' & b' \\ c' & d' \end{smallmatrix}\right) &  \mapsto (c',d')   
\end{split}
\end{equation}

Par un argument standard de la théorie de la déformation, on a le résultat suivant (voir \cite[Lemme 2.3]{D-B}).
 
\begin{lemma} \label{lemmatd}\
\begin{enumerate}
  
 \item $t_{\cD^{ord}}= \ker \left( \rH^1(F,\ad \rho) \overset{(C^{*}_i,D^{*}_i)}{\rightarrow}  \prod_{\gp_i \mid p} (\rH^1(F_{\gp_i},\psi_{i}''/\psi_{i}') \oplus \rH^1(I_{\gp_i}, \psi_{i}''/\psi_{i}'')) \right)$

 \item $t^{0}_{\cD^{ord}}= \ker \left( \rH^1(F,\ad^0 \rho) \overset{(C^{*}_i,D^{*}_i)}{\rightarrow} \prod_{\gp_i \mid p} (\rH^1(F_{\gp_i},\psi_{i}''/\psi_{i}') \oplus \rH^1(I_{\gp_i}, \psi_{i}''/\psi_{i}'')) \right)$ 

 \item $t_{\cD}= \ker \left( \rH^1(F,\ad \rho) \overset{C^{*}_i}{\rightarrow}  \prod_{\gp_i \mid p}( \rH^1(F_{\gp_i},\psi_{i}''/\psi_{i}')) \right)$
 
 \end{enumerate} 
 
\end{lemma}   

\subsection{La suite exacte inflation-restriction appliquée à $t_\cD$}\

On note $H\subset \bar\Q$ le corps de nombres totalement complexe fixé par $\ker(\ad \rho)$. Le groupe $G'=\Gal(H/F)$ est naturellement isomorphe à l'image projective de $\rho$ qui est un groupe diédral. Pour tout $\gp_i$ de $F$ au dessus de $p$, on note $w_i$ la place canonique de $H$ au dessus de $\gp_i$ induite par $\iota_p$.

\begin{lemma}\cite[Lemme 2.4]{D-B} \label{Rest-Inflation}
Soient $L$ un corps de nombres et $\rho$ une représentation d'image finie de $G_{L}$ dans un $\bar{\Q}_p$-espace vectoriel de dimension finie.
\begin{enumerate}
\item Soit $\Sigma$ un ensemble de places de $L$ et supposons que pour tout $v \in \Sigma $ il existe un quotient $\rho_v$ de $\rho_{|G_{L_v}}$. Soit $H$ un extension galoisienne finie de $L$, pour tout $v\in \Sigma $, on choisit une place $w(v)$ de $H$ au dessus de $v$. Alors 
$$\ker \left( \rH^1(L,\rho) \rightarrow \prod_{v\in \Sigma} \rH^1(L_{v},\rho_v)
 \right) \simeq \ker \left(\rH^1(H,\rho)^{\Gal(H/L)} \rightarrow \prod_{v\in \Sigma} \rH^1(H_{w(v)},\rho_v) \right).$$
 
\item Le morphisme naturel $$\rH^1(L, \rho)  \rightarrow \prod_{v \mid p} \rH^1(I_v,\rho)$$ est injectif.
 \end{enumerate}
 \end{lemma}

En utilisant la base $\{e_1,e_2\}$ définie ci-dessus, on peut voir les éléments de $\rH^1(F,\ad \rho)$ comme des $1$-co-cycles
\begin{equation}
\begin{split}
G_{F}\rightarrow M_2(\bar{\Q}_p) \\ 
\enspace g\mapsto \left(\begin{smallmatrix}a(g) & b(g) \\ c(g) & d(g) \end{smallmatrix}\right)
\end{split}
\end{equation}

Puisque $\rho$ est irréductible, $\psi^\sigma \neq \psi$ et $\ker(\psi^\sigma/\psi)$ définit une extension cyclique $H$ de $M$. Ainsi, on a la décomposition suivante $\bar{\Q}_p[G_F]$-modules:

\begin{equation}\label{13}
\ad\rho\simeq1 \oplus \ad^0\rho\simeq 1 \oplus\ad^0( \Ind_M^{F} \psi)\simeq 
1\oplus \epsilon_M \oplus
\Ind_M^{F} (\psim),
\end{equation}

Par inflation-restriction, on a aussi
$$\rH^1(F,\ad \rho)=\rH^1(M,\ad \rho)^{\Gal(M/F)}.$$
 
D'autre part, on a la décomposition suivante 

\begin{equation}\label{adbc}
\begin{split}
\rH^1(M,\ad \rho)& \simeq \rH^1(M,\psi/\psi)  \oplus \rH^1(M,\psi_{_{\heartsuit}})\oplus  \rH^1(M,\psi_{_{\heartsuit}}^{-1}) \oplus \rH^1(M,\psi^{\sigma}/ \psi^{\sigma}), \\
\left(\begin{smallmatrix}a & b \\ c & d \end{smallmatrix}\right)& \mapsto (a,b,c,d),
\end{split}
\end{equation}

où l'action du groupe $\Gal(M/F)$ échange $a$ et $d$ et aussi $b$ et $c$.

En combinant le lemme \ref{lemmatd} et le lemme \ref{Rest-Inflation}, on obtient: 

\begin{cor}\label{in tg} Dans la base $(e_1,e_2)$, on a:

\begin{enumerate}

\item 
${ \small t_{\cD^{ord}}\simeq \ker \left(\rH^1(M,\ad \rho)^{\Gal(M/F)} \overset{(C^*_{i},D^*_{i})}{\rightarrow}  \underset{ \gp_i \mid p}{\prod} \rH^1(M_{v_i},\psi_{i}''/\psi_{i}') \underset{ \gp_i \mid p}{\prod} \rH^1(I_{v_i}, \psi_{i}''/\psi_{i}'') \right)} $

\item Si $\left(\begin{smallmatrix} a & b \\ c & d \end{smallmatrix}\right)\in t_{\cD} \subset \rH^1(M,\ad \rho)^{\Gal(M/F)}$, alors $a=d^{\sigma}$ et $b=c^{\sigma}$.

\end{enumerate}

\end{cor}

L'action naturelle à gauche de $\Gal(H/F)$ sur $G_H$ induit une action à droite donnée par $x\rightarrow g^{-1}(x)$ pour laquelle $\rH^{1}(H,\bar{\Q}_p)$ est un $\bar{\Q}_p[\Gal(H/F)]$-module à gauche.

On a $\rH^{1}(H,\ad \rho)=\rH^{1}(H,\bar{\Q}_p) \otimes_{\bar{\Q}_p} \ad \rho$ (puisque $\ad \rho(G_H)=1$). 

Un élément de $\rH^{1}(H,\bar{\Q}_p) \otimes_{\bar{\Q}_p} \ad \rho$ peut être écrit comme une matrice $\left(\begin{smallmatrix} a & b \\ c & d \end{smallmatrix}\right)$, où $a,b,c$ et $d$ sont des éléments de $\Hom(G_H, \bar{\Q}_p)$ et l'action naturelle à gauche de $\Gal(H/F)$ sur $\rH^{1}(H,\bar{\Q}_p) \otimes_{\bar{\Q}_p} \ad \rho$ est donnée par 

$$g.\left(\begin{matrix} a & b \\ c & d \end{matrix}\right)=\rho(g) \left(\begin{matrix} g.a & g.b \\ g.c & g.d \end{matrix}\right) \rho(g)^{-1}.$$

Ainsi, l'inflation restriction appliquée à l'extension galoisienne finie $H/F$ induit l'isomorphisme suivant: 
\begin{equation*}
\rH^1(F,\ad \rho)\simeq\left(\Hom(G_H, \bar{\Q}_p)\otimes \ad \rho\right)^{G'}.
\end{equation*}

\begin{cor} \label{in tg n,ord} Dans la base $(e_1,e_2)$, on a:

$ { \small t_{\cD} \simeq \ker \left(\left(\Hom(G_H, \bar{\Q}_p)\otimes \ad \rho\right)^{G'} \overset{(C^*_{i})}{\rightarrow} \underset{w_i, \gp_i \mid p}{\prod} \rH^1(H_{w_i}, \bar{\Q}_p)) \right)}$

\end{cor}

\begin{proof}
Le résultat découle des lemmes \ref{lemmatd} et \ref{Rest-Inflation}

\end{proof}

\section{Application de la th\'eorie des corps de classes}\label{local unit p-adic}\

\subsection{Rappel sur les unit\'es locales et globales}\

Si $G$ est un groupe fini, on notera $\widehat{G}$ l'ensemble des classes d'isomorphismes de représentations irréductibles de $G$ sur $ \bar{\Q}$.

Soient $\cO_H$ l'anneau des entiers de $H$, $M_p$ la $p$-extension abélienne maximale non ramifiée en dehors de $p$ de $H$ et $\cO_{H_{w}}$ l'anneau des entiers de la complétion $H$ par rapport à $w$, où $w$ est la place de $H$ au dessus de $p$. Soit $H''$ le $p$-Hilbert de $H$, alors la théorie du corps de classes implique que $$\Gal(M_p/H'')\simeq \text{la }p\text{-partie de } (\cO_{H}\otimes \Z_p)^\times/ \bar\cO_{H}^\times \simeq U^{1}_{p}/\bar\cO_{H}^\times,$$ où $U^1_{p}=\prod_{{w} \mid p} U^1_{w}$ et $U^{1}_{w}$ est le groupe des unités de $\cO_{H_w}$ qui sont $\equiv 1 \pmod p$. Ainsi, on a la suite exacte suivante:

\begin{equation}\label{global-local}
0\longrightarrow \Hom(G_H,\bar \Q_p) \longrightarrow  \Hom \left( (\cO_{H}\otimes \Z_p)^\times, \bar \Q_p\right)\longrightarrow \Hom \left(\cO_{H}^{\times},\bar{\Q}_p \right),
\end{equation}

telle que la première application est le dual de la réciprocité d'Artin et la seconde est la restriction à $\cO_{H}^{\times}$ par rapport à l'inclusion diagonale $\cO_{H}^{\times}\hookrightarrow \left(\cO_H \otimes \Z_p \right)^{\times}$.

Tout morphisme continu pour la topologie $p$-adique de $\Hom (\cO_{H,{w}}^\times, \bar \Q_p)$ est donné par 
\begin{equation}\label{log-loc}
u\mapsto 
\sum_{g_w\in J_w} h_{g_w} g_w(\log_p(u))=
\sum_{g_w\in J_w} h_{g_w} \log_p(g_w(u)),
\end{equation}
où $h_{g_w}$ est un élément de $\bar \Q_p$, $J_w$ est l'ensemble des plongements de $H_w$ dans $\bar{\Q}_p$, et $\log_p$ est logarithme $p$-adique définie sur $\bar \Q_p^\times$.

On note $J_H$ l'ensemble des plongements de $H$ dans $\bar{\Q}$. Les deux plongements $\iota_p:\bar\Q\hookrightarrow\bar\Q_p$ et $H\subset \bar\Q$ définissent une partition $J_H=\coprod_{w\mid p} J_w$ provenant du diagramme commutatif suivant
\begin{equation}\label{place H}
\xymatrix{ \bar\Q \ar@{^{(}->}^{\iota_p}[r] & \bar\Q_p \\
H\ar@{^{(}->}^{g}[u] \ar@{^{(}->}[r] & H_w. \ar@{^{(}->}^{g_w}[u] }\end{equation}

Dans l'introduction, on avait fixé des plongements complexes $g_1,...,g_n$ de $\bar{\Q}$ dans $\bar{\Q}$ qui relèvent les plongements de $F$ dans $\bar{\Q}$ et tels que $g_1$ est le plongement qui identifie $F$ avec un sous-corps de $\R$. Donc on obtient la partition suivante:

\begin{equation}\label{JH}
J_H=\{g_i\circ g | 1\leq i \leq n, g\in \Gal(H/F)\}=\coprod_{1\leq i \leq n} g_i.\Gal(H/F)
\end{equation}

Ainsi, les éléments définis par 

$\left(u\otimes 1 \mapsto \log_p \big(\iota_p \circ g_i \circ g(u) \big)\right)$ pour $1\leq i \leq n$ et $g\in \Gal(H/F)$ forment une base du $\bar{\Q}_p$-espace vectoriel $\Hom \left( (\cO_{H}\otimes \Z_p)^\times, \bar{\Q}_p \right)$. On a une action naturelle à gauche de $\Gal(H/F)$ sur $\Hom \left( (\cO_{H}\otimes \Z_p)^\times, \bar{\Q}_p \right)$ donnée par $$g'.\log_p \big(\iota_p \circ g_i \circ g \otimes 1) = \log_p \big(\iota_p \circ g_i \circ g'g \otimes 1).$$

Par conséquent, il existe un isomorphisme canonique de $G'=\Gal(H/F)$-modules à gauche: 

\begin{equation}\label{log_p}
\begin{split}
\underset{1\leq i \leq n}\oplus \bar{\Q}_p[G'] & \overset{\sim}{\longrightarrow}  \Hom \left( (\cO_{H}\otimes \Z_p)^\times, \bar \Q_p\right) \\
\left(\sum_{g\in G' }  a_{i,g}  g\right)_{1\leq i \leq n}  & \mapsto
\left(u\otimes 1 \mapsto \sum_{g \in G' \atop 1\leq i \leq n} a_{i,g} \log_p \big(\iota_p\circ g_i \circ g^{-1}(u) \big) \right).
\end{split}
\end{equation}

On a les conjugaisons complexes $\tau_1,..., \tau_n \in G_F$ telles que $\tau_i$ est la conjugaison complexe attachée à $g_i$ (i.e $\tau_i$ est conjuguée à la conjugaison complexe $c$ de $G_\Q$ par $g_{i}$ et $c=\tau_1$). Puisque $\rho$ est impaire, alors $\det \rho(\tau_i)=-1$ pour tout $1\leq i \leq n$.

Soit $\{\sigma^{i}_j\}_{1\leq j \leq n'}$ un ensemble de représentants de toutes les classes d'équivalence du quotient de l'ensemble $\{g_i\circ g | g\in \Gal(H/F), 1\leq i \leq n\}$ par la relation d'équivalence $$ c \circ h \sim h \text{, où $h \in g_i G'$}.$$ D'après la preuve de Minkowski du théorème des unités de Dirichlet, il existe un plongement $\cO_H^{\times}/ \mu \hookrightarrow \R^{nn'}$ donné par $a \rightarrow (\log|\sigma^{i}_j(a)|)_{1\leq j \leq n', 1\leq i \leq n}$, où $\mu$ est le sous-groupe de torsion de $\cO_H^{\times}$. 

D'autre part, on a $g_i(\tau_i(x))=\overline{(g_i(x))}$ et donc le $G'$-ensemble $\{\sigma^{i}_j\}_{1\leq j \leq n'}$ peut être vu comme la permutation des classes à droite du groupe $\Gal(H/F)$ suivant le sous-groupe $\{1,\tau_i\}$.

Ainsi, on obtient les décompositions de $G'$-modules: 

$$\Hom(\cO_{H}^{\times},\bar\Q_p) \simeq  (\bigoplus_{\pi \ne 1 \atop \pi \in \widehat{G'}} \pi^{ \alpha_\pi })\oplus 1^{n-1}  \text{ et  } \Hom \left( (\cO_{H}\otimes \Z_p)^\times, \bar \Q_p\right)   \simeq  \bigoplus_{\pi \in \widehat{G'}} \pi^{n\dim \pi}$$

où $\pi$ parcourt l'ensemble des caractères de $G'=\Gal(H/F)$, $\alpha_\pi= \sum\limits^{n}_{i=1} \dim \pi^{+\tau_i}$ et les $\pi^{+\tau_i}$, $\pi^{-\tau_i}$ sont respectivement les espaces propres associés à l'action de la conjugaison complexe $\tau_i \in G'$ pour les valeurs propres $+1$ et $-1$. 

D'autre part, on la décomposition de $G'$-modules: $$\Hom(G_H, \bar{\Q}_p)\simeq \bigoplus\limits_{\pi \in \widehat{G'}} \pi ^{m_\pi}.$$

En utilisant la suite exacte (\ref{global-local}), on obtient une borne de $m_{\pi}$:

\begin{lemma} \label{multiplicity-Mix} 
On a  $m_1 \geq 1$ et pour $\pi\neq 1$, on a $ \sum\limits^{n}_{i=1} \dim \pi^{-\tau_i} \leq  m_\pi$ avec égalité si la conjecture de Leopoldt est vraie pour $H$. 

\end{lemma}

\subsection{Une minoration de $ \dim t^0_{\cD^{ord}}$}\label{M is mixt}\

L'inflation-restriction induit l'isomorphisme suivant

\begin{equation}\label{iso-inf}
\rH^1(M,\psi_{_{\heartsuit}}^{-1})\simeq \rH^1(H,\bar \Q_p)[\psi_{_{\heartsuit}}^{-1}]
\end{equation}
où $V[\psi_{_{\heartsuit}}^{-1}]$ est le $\psi_{_{\heartsuit}}^{-1}$-sous-espace propre de la $\Gal(H/M)$-représentation $V$.

\begin{prop} \label{dimension neither nor}

Supposons que $M$ a $2r$ plongements complexes et le caractère $\psi^2_{_{\heartsuit}}$ est non trivial. Alors:

 $$2[F:\Q]-r \leq \dim \rH^1(M,\psi_{_{\heartsuit}}^{-1}).$$

\end{prop}
\begin{proof}  

Puisque le caractère $\psi^2_{_{\heartsuit}}$ est non trivial, $\pi=\Ind_M^{F}(\psi_{_{\heartsuit}}^{-1})$ est irréductible, donc : 
$$\dim \rH^1(H,\bar \Q_p)[\psi_{_{\heartsuit}}^{-1}]=
\dim \Hom_{G'}\left(\pi, \rH^1(H,\bar\Q_p)\right)=m_ \pi.$$

On note $(\tau_{i_k})_{1\leq k \leq r}$ (resp.$(\tau_{i'_{l}})$) les conjugaisons complexes de $G_F$ qui se prolongent en plongements complexes (resp. réels) de $M$ dans $\bar{\Q}$. On voit que $\det(\pi)(\tau_{i_{k}})=-1$, $\det(\pi) (\tau_{i'_{l}})=1$, $\dim \pi^{-\tau_{i_{k}}}=1$ et $\dim \pi^{-\tau_{i'_{l}}}=2$.

D'après le lemme \ref{multiplicity-Mix}, on a $ 2[F:\Q]-r \leq m_\pi$, et on conclut par l'isomorphisme $(\ref{iso-inf})$. 

\end{proof}

\begin{lemma}\label{iso loc}
Soit $\gp_i$ un premier de $F$ au dessus de $p$ qui se décompose dans $M$ en $v_i$ et $v_i^{\sigma}$, alors $\dim \rH^{1}(M_{v_i}, \psi_{_{\heartsuit}}^{-1})= e_i . f_i$.

\end{lemma}

\begin{proof}

La $p$-régularité de $\rho$ implique que $\rH^0(M_{v_i}, \psi_{_{\heartsuit}}^{-1})=0$. D'après la dualité locale de Tate,
$$\rH^2(M_{v_i}, \psi_{_{\heartsuit}}^{-1})\simeq \rH^0(M_{v_i}, \psi_{_{\heartsuit}}^{-1}(1))^\vee = 0.$$

Finalement, la formule de Tate de la caractéristique d'Euler locale implique
$$\dim \rH^1(M_{v_i}, \psi_{_{\heartsuit}}^{-1}) = [F_{\gp_i}:\Q_p]=e_i . f_i.$$

\end{proof}

On note $I \subset S_p$ (resp. $I' \subset S_p$) l'ensemble des plongements qui induisent des premiers $\gp_i \mid p$ de $F$ tels que $\psi^{\sigma}(\Frob_{v_i})= \alpha_i$ (resp. $\psi(\Frob_{v_i})= \alpha_i$) où $U_{\gp_i}.f=\alpha_i f$.

\begin{rem}

La base $(e_1, e_2)$ telle que $\rho_{|G_M}=\psi \oplus \psi^{\sigma}$ est définie à des scalaires près $\left(\text{i.e } (e_1,e_2) \sim (\mu e_1,\mu' e_2) \text{ où }(\mu,\mu') \in (\bar{\Q}^{\times})^2\right)$. Si $\left(\begin{smallmatrix} a & b \\ c & d \end{smallmatrix}\right)\in \rH^1(M,\ad \rho)^{\Gal(M/F)}$ dans la base $(e_1,e_2)$, alors un changement de la base $(e_1,e_2)$ à des scalaires près ne modifiera pas $a$ et $d$ et modifiera $c$ et $b$ par multiplication de scalaires. 

\end{rem}

Si $\gp_i$ est inerte ou ramifié dans $M$, alors par ordinarité en $p$, $\psi$ est non ramifié en $v_i$ et $\psi_{|G_{M_{v_i}}} =\psi^\sigma_{|G_{M_{v_i}}}$. Il en découle que $v_i$ se décompose complètement dans $H$ et que $\psi_{i}''/ \psi_{i}'$ est le caractère quadratique de $\Gal(M_{v_i}/F_{\gp_i})$. Ainsi, si $\gp_i \in S^p$, en normalisant la base $(e_1,e_2)$ de sorte que  $\rho(\sigma_0)=\left(\begin{smallmatrix}0 & \lambda \\ \lambda & 0 \end{smallmatrix}\right)$, où $\sigma_0$ est un élément fixé non trivial du groupe $\Gal(H_{w_i}/F_{\gp_i})$. On notera que $\rho(\sigma_0)$ est diagonal dans la base $(e_{1}+e_{2}, e_1-e_2)$. Donc on peut supposer que $(e_{1}+e_{2}, e_1-e_2)=(e'_{i,1},e'_{i,2})$ (voir \cite[\S4]{D-B}). En effet, avec les notations de la section \S \ref{Tangent space}, un calcul direct montre que le morphisme de la relation $(\ref{locp})$ est donné dans la base $(e_1,e_2)$ par

$$  \left(\begin{smallmatrix}a & b \\ c & d \end{smallmatrix}\right)   \overset{(C^*_{i},D^*_{i})}{\longrightarrow}   { \small  (\frac{a-c+b-d}{2},\frac{ a-c-b+d}{2} )} $$

De plus $\sigma_0$ échange $b_{|G_M}$ et $c_{|G_M}$. Ainsi, si
$\left(\begin{smallmatrix}a & b \\ c & d \end{smallmatrix}\right)\in t_{\cD^{ord}}$, alors 
\begin{equation}\label{non-split}
c^{\sigma_0}_{v_i}-a^{\sigma_0}_{v_i}=c_{v_i}-a_{v_i}  \text{ et }  a^{\sigma_0}_{v_i}+a_{v_i}-c^{\sigma_0}_{v_i}-c_{v_i}=0 
\end{equation}
avec $c_{v_i}=i(c)$ et $a_{v_i}=j(a)$, où $i$ et $j$ sont les restrictions suivantes: 
\begin{equation}\label{restriction}
c\in \rH^1(M, \psi_{_{\heartsuit}}^{-1})\overset{i}{\rightarrow} 
\mathrm{im}\left(\rH^1(M_{v_i},\bar \Q_p) \rightarrow \rH^1(I_{v_i},\bar \Q_p)\right)
\overset{j}{\leftarrow} \rH^1(M,\bar \Q_p) \ni a.
\end{equation}

Pour calculer $t_{\cD^{ord}}^{0}$, on doit rajouter la condition $a+d=0$ qui est équivalente à $d^\sigma=-d$.

\begin{prop}\label{tg M mixt}

On suppose que:
\begin{enumerate}

\item Le corps $M$ a $2r$ plongements complexes. 
\item Le caractère $\psi^2_{_{\heartsuit}}$ est non trivial.
\end{enumerate}

Alors $\dim t^0_{\cD^{ord}} \geq [F:\Q]-r-|S^p|-\sum_{\substack{\gp_i \in S^p}} f_i . e_i.$.

\end{prop}

\begin{proof}

Soit $\left(\begin{smallmatrix} a & b \\ c & d \end{smallmatrix}\right)$ un $1$-cocycle de $t^0_{\cD^{ord}}$ dans la base $(e_1,e_2)$. D'après le corollaire \ref{in tg}, on a $a=d^{\sigma}$. De plus, si $a=d=0$, alors les équations (\ref{non-split}) impliquent que $c_{v_i}=0$ pour tout premier $\gp_i$ de $F$ inerte ou ramifié dans $M$.

Puisque $S_p=I \bigsqcup I'$, le corollaire \ref{in tg n,ord} et les équations (\ref{non-split}) impliquent que $c$ (resp. $b$) est non ramifié en $\gp_i \in I$ (resp. $\gp_i \in I'$) et puisque $\sigma$ échange $b$ et $c$, on obtient l'inclusion suivante
\begin{equation}\label{CM relative}
\ker \left(\rH^1(M,\psi_{_{\heartsuit}}^{-1}) \overset{(C^*_{i})}{\rightarrow}  \underset{ \gp_i \in I}{\prod} \rH^1(M_{v_i},\psi_{_{\heartsuit}}^{-1}) \underset{ \gp_i \in I'}{\prod} \rH^1(M_{v_i^{\sigma}},\psi_{_{\heartsuit}}^{-1}) \underset{ \gp_i \in S^p}{\prod} \rH^1(M_{v_i},\bar{\Q}_p) \right) \subset t^0_{\cD^{ord}}\end{equation}
D'après les lemmes \ref{dimension neither nor}, \ref{iso loc} et le fait que $\dim_{\bar{\Q}_p}  \Hom(G_{M_{v_i}},\bar{\Q}_p)=2e_i.f_i+1$ si $\gp_i$ est inerte où ramifié dans $M$, on a la minoration suivante  $$\dim t^0_{\cD^{ord}} \geq [F:\Q]-r-|S^p|-\sum_{\substack{\gp_i \in S^p}} f_i . e_i.$$

\end{proof}

\section{Calcul de la dimension des espaces tangents de $\cD$, $\cD^{ord}$ et $\cD^{ord}_{\det \rho}$}\label{M is Real}\

Dans cette section, on calcule les dimensions des espaces tangents des foncteurs $\cD$, $\cD^{ord}$ et $\cD^{ord}_{\det \rho}$ dans le cas où $M$ est totalement réel ou totalement complexe sous certaines conditions sur $\rho$ pour ce dernier cas.

\subsection{Le cas où $M$ est totalement réel}\

On suppose dans cette sous-section que $M$ est totalement réel et on considère la condition suivante sur $M$:

\begin{itemize}
\item[$\bullet$]   $({ \bf L}_M)$ la conjecture de Leopoldt est vraie pour $M$.

\end{itemize}

Puisque $\det \rho(\tau_i)=-1$ pour tout $1 \leq i \leq n$ (i.e $\rho$ est totalement impaire), alors $\psi_{_{\heartsuit}}^{-1}(\tau_i)=-1$ pour les conjugaisons complexes $(\tau_i)_{\{1\leq i \leq n\}}$ de $G_F$. 

En conséquence, les conjugaisons complexes $(\tau_i)_{\{1\leq i \leq n\}}$ définissent le même élément $c$ du groupe $\Gal(H/M)$ et $c$ est dans le centre de $\Gal(H/F)$. Ainsi, le sous-corps $H_+$ de $H$ fixé par $c$ est totalement réel et $H$ est un corps CM.

En généralisant les techniques de \cite[théorème 3.1]{cho-vatsal}, on obtient la proposition suivante.

\begin{prop}\label{factoring global}\
\begin{enumerate}

\item Il existe un isomorphisme $$\Hom(\prod_{{w} \mid p} U^1_{w}/\bar\cO_{H}^\times, \bar{\Q}_p )[\psi_{_{\heartsuit}}^{-1}]\simeq \rH^{1}(M, \psi_{_{\heartsuit}}^{-1}).$$

\item La projection canonique $\prod_{{w} \mid p}U^1_{w}  \twoheadrightarrow \prod_{{w} \mid p}U^1_{w}/\bar\cO_{H}^\times$ induit un isomorphisme entre les $\psi_{_{\heartsuit}}^{-1}$-espaces propres $$\Hom(\prod_{{w} \mid p} U^1_{w}, \bar{\Q}_p)[\psi_{_{\heartsuit}}^{-1}]\simeq \Hom(\prod_{{w} \mid p} U^1_{w}/\bar\cO_{H}^\times, \bar{\Q}_p)[\psi_{_{\heartsuit}}^{-1}].$$

\end{enumerate}
\end{prop}

\begin{proof}\

(i)D'après la suite exacte (\ref{global-local}) et l'isomorphisme (\ref{iso-inf}), on a $$\rH^1(M,\psi_{_{\heartsuit}}^{-1})\simeq \Hom(\prod_{{w} \mid p} U^1_{w}/\bar\cO_{H}^{\times}, \bar{\Q}_p)[\psi_{_{\heartsuit}}^{-1}].$$

(ii)Soit $\varrho$ un élément de $\Hom(\prod_{{w} \mid p} U^1_{w}, \bar{\Q}_p)[\psi_{_{\heartsuit}}^{-1}]$. Pour tout $g$ dans $\prod_{{w} \mid p} U^1_{w}$, on a:
\begin{equation}\label{Glob}
\begin{split}
\varrho(g) &=\psi_{_{\heartsuit}}^{-1}(c) \varrho(c(g))
\\&= -\varrho(c(g))
\end{split}
\end{equation} 

De plus, $H$ est un corps CM et donc le groupe quotient $\cO_{H}^{\times}/\cO_{H_{+}}^{\times}$ est d'ordre fini ($H_+=H^{c}\subset \R$). Maintenant, on peut voir que la relation $(\ref{Glob})$ implique que $\forall x \in \cO_{H_{+}}^{\times}$, on a $\varrho(x) = - \varrho(c(x))= - \varrho(x)$, donc $\varrho$ se factorise sur $\cO_{H_+}^{\times}$ et puisque le groupe $\cO_{H}^{\times}/\cO_{H_+}^{\times}$ est d'ordre fini, $\varrho$ se factorise sur $\cO_{H}^{\times}$. De là, on obtient l'isomorphisme voulu.

\end{proof}

\begin{prop}\label{tg ord}Avec les notations de (\ref{non-split}), on a:

\begin{enumerate}
 
\item Si $\left(\begin{smallmatrix} a & b \\ c & d \end{smallmatrix}\right)\in t_{\cD}\subset \rH^1(M,\ad \rho)^{\Gal(M/F)}$, alors $a=a^{\sigma}$.

\item Soit $\gp_i$ un premier de $F$ qui est inerte ou ramifié dans $M$, alors $c_{v_i}=c_{v_i}^{\sigma_0}$. 

\item Soient $\left(\begin{smallmatrix} a & b \\ c & d \end{smallmatrix}\right)\in t^0_{\cD^{ord}}\subset \rH^1(M,\ad^0 \rho)^{\Gal(M/F)}$ et $\gp_i$ un premier de $F$ qui est inerte ou ramifié dans $M$, alors $c_{v_i}$ et $a_{v_i}$ sont non ramifiés.

\end{enumerate}

\end{prop}

\begin{proof}

(i) L'hypothèse $({ \bf L}_M)$ implique que $M.\Q_{\infty}$ est l'unique $\Z_p$-extension de $M$ et donc $a=a^{\sigma}$.

(ii) Il découle de la proposition \ref{in tg n,ord}, de (i) et de (\ref{non-split}) que $c_{v_i}=c_{v_i}^{\sigma_0}$. 

(iii) L'énoncé (i) implique que $d=0$, par conséquent, (ii) et (\ref{non-split}) impliquent que $c_{v_{i}}$ est trivial.
\end{proof}

\begin{theorem}\label{dim c}On suppose que $M$ est totalement réel, alors:

\begin{enumerate}

\item La dimension de $t_{\cD}$ est $[F:\Q]+1$ et la dimension de $t_{\cD^{ord}}$ est $\max \{1, \sum_{\substack{\gp_i \in S_p}} f_i . e_i\}$.

\item On a l'isomorphisme suivant 

\begin{equation*}\label{1-co-cycles}
\begin{split}
& t^0_{\cD^{ord}}\simeq  \Hom(\prod_{w \mid v_i^{\sigma},\gp_i \in I} U^1_{w} \prod_{w \mid v_i,\gp_i \in I' } U^1_{w}, \bar{\Q}_p)[\psi_{_{\heartsuit}}^{-1}] \\
& \text{et } \dim t^0_{\cD^{ord}}= \sum_{\substack{\gp_i \in S_p}} f_i . e_i.
\end{split}
\end{equation*}
\end{enumerate}
\end{theorem}

\begin{proof}

i) Soit $\left(\begin{smallmatrix} a & b \\ c & d \end{smallmatrix}\right)$ un élément de $t_{\cD}\subset \rH^1(F,\ad \rho)^{\Gal(H/F)}$. D'après la décomposition (\ref{adbc}) et le corollaire \ref{in tg}, il suffit de déterminer la dimension des espaces vectoriels où vivent $d$ et $c$. 

La décomposition (\ref{adbc}) implique que $d\in \Hom(G_M,\bar{\Q}_p)$ et puisque la condition $({ \bf L}_M)$ est vérifiée, alors $\dim \Hom(G_M,\bar{\Q}_p)=1$. 

D'autre part, $\sigma$ échange $b$ et $c$ et échange aussi $v_i^{\sigma}$ et $v_i$ quand $\gp_i \in S_p=I' \bigsqcup I$. D'après l'isomorphisme (\ref{log_p}), les propositions \ref{in tg n,ord}, \ref{factoring global}, \ref{tg ord} et le fait que $\Gal(H/M)$ permute les places de $H$ au dessus de $v_i$, on a:

{ \small 
\begin{equation}\label{c iso}
\begin{split} 
& c \in  \Hom(\prod_{w \mid v_i^{\sigma},\gp_i \in I} U^1_{w} \prod_{w \mid v_i,\gp_i \in I' } U^1_{w}, \bar{\Q}_p)[\psi_{_{\heartsuit}}^{-1}] \oplus  (\Hom(\prod_{w \mid v_i,\gp_i \in S^p } U^1_{w}, \bar{\Q}_p)[\psi_{_{\heartsuit}}^{-1}])^{\Gal(M_{v_i}/F_{\gp_i})} \\
\dim& \Hom(\prod_{w \mid v_i^{\sigma},\gp_i \in I} U^1_{w} \prod_{w \mid v_i,\gp_i \in I' } U^1_{w}, \bar{\Q}_p)[\psi_{_{\heartsuit}}^{-1}] \oplus  (\Hom(\prod_{w \mid v_i,\gp_i \in S^p } U^1_{w}, \bar{\Q}_p)[\psi_{_{\heartsuit}}^{-1}])^{\Gal(M_{v_i}/F_{\gp_i})}=[F:\Q]
\end{split}
\end{equation}

}
Donc, $\dim t_{\cD}= [F:\Q]+1$.

Si on suppose de plus que $\left(\begin{smallmatrix} a & b \\ c & d \end{smallmatrix}\right)\in t_{\cD^{ord}}$, on doit ajouter la condition supplémentaire (\ref{non-split}) quand $\gp_i \in S^p$ et $d_{v_i}$ est non ramifié quand $\gp_i \in S_p$. On procède cas par cas: 

Si $S_p \ne \varnothing$, alors $d$ est non ramifié en une place de $M$ au dessus de $p$ et il en découle que $d$ est trivial (puisque on a supposé la condition $({ \bf L}_M)$). Ainsi, la proposition \ref{tg ord} et  la relation (\ref{non-split}) impliquent que $c_{v_i}$ est trivial quand $\gp_i \in S^p$. Donc, (\ref{c iso}) implique que $\dim t_{\cD^{ord}}=|I'_F | =\sum_{\substack{\gp_i \in S_p}} f_i . e_i$.

Si $S_p$ est vide, les relations (\ref{non-split}) et $d\in \Hom(G_M,\bar{\Q})$ impliquent que $\dim t_{\cD^{ord}}=1$.

ii) Finalement, si on suppose que $\left(\begin{smallmatrix} a & b \\ c & d \end{smallmatrix}\right)\in t^0_{\cD^{ord}}$, alors les relations (\ref{c iso}) et (\ref{non-split}) et la proposition \ref{tg ord} impliquent l'isomorphisme désiré et donc $\dim t^0_{\cD^{ord}}=|I'_F|= \sum_{\substack{\gp_i \in S_p}} f_i . e_i$.

\end{proof}

\subsection{Le cas où $H$ est Galois et diédral sur $\Q$.}\label{M is complex}\

On a sous les hypothèses du Théorème \ref{CM-case} le diagramme suivant

\begin{equation}
\centering
\xymatrix{&H\ar@{-}[d]\\
&F.K=M\ar@{-}[dl]\ar@{-}[dr]\\
 K\ar@{-}[dr]&&F\ar@{-}[dl]\\
  &\Q }
 \end{equation}
  
Le groupe de Galois $\Gal(H/\Q)$ est diédral et les extensions $M/F$, $M/K$ et $K/\Q$ sont quadratiques ($K$ est imaginaire).

La restriction $\Gal(M/\Q)\rightarrow \Gal(K/\Q)$ induit un isomorphisme $\Gal(M/F) \simeq \Gal(K/\Q)$ et donc l'ensemble $S^p$ est vide (puisque $p$ se décompose dans $K$).

On note $G$ le groupe de Galois $\Gal(H/\Q)$. Avec les notations utilisées dans la relation (\ref{log_p}), tout élément de $\Hom ((\cO_{H}\otimes \Z_p)^\times, \bar{\Q}_p)$ est de la forme $$u\otimes 1 \mapsto \sum_{g \in G} h_{g} \log_p(g^{-1}(u\otimes 1))$$ pour $h_{g}\in  \bar \Q_p$. De plus, la $G$-représentation régulière $\bar{\Q}_p[G]=\Hom \left((\Z_p \otimes \cO_{H})^{\times}, \bar{\Q}_p \right)$ est isomorphe à $\bigoplus\limits_{\pi \in \widehat{G}} \pi^{\dim_\pi}$, où $\pi$ parcourt l'ensemble des caractères du groupe $G$. 

D'autre part, on a la décomposition de $G$-modules:

\begin{equation}\label{mutiplicity-}
\rH^1(H, \bar{\Q}_p) \simeq \bigoplus_{\pi} \pi^{m_\pi},
\end{equation}
où $\pi$ parcourt l'ensemble des caractères de $G$. D'après la preuve de Minkowski du théorème des unités de Dirichlet, il existe un isomorphisme de $G$-modules $$\Hom(\cO_{H}^{\times}, \bar{\Q}_p)\simeq \bigoplus\limits_{\pi \in \widehat{G}} \pi ^{\dim \pi^{+}},$$ où $\pi^+$ et  $\pi^-$ sont les sous-espaces propres pour l'action de la conjugaison complexe $c \in G$.

\begin{lemma} \label{multiplicity}Avec les notations ci-dessus, on a:

\begin{enumerate}

\item $m_\pi=\dim \pi^{-}$ pour $\pi \neq 1$ et $m_1=1$.

\item $\dim \pi \leq 2$ pour tout caractère $\pi$ de $G$. De plus, si $\dim \pi =2$, alors $\pi= \Ind^\Q_K \psi$, où $\psi:G_K \rightarrow \bar{\Q}_p^{\times}$ est un caractère. 

\end{enumerate}
\end{lemma}

\begin{proof}\
(i) Puisque $K$ est un corps quadratique imaginaire et $H$ est une extension abélienne de $K$, la conjecture de Leopoldt est vraie pour $H$ d'après Ax-Baker-Brumer \cite{brumer}. Par conséquent, le dernier morphisme de la suite exacte (\ref{global-local}) est surjectif et donc $m_{\pi}=\dim \pi^{-}$. 

(ii) L'assertion découle immédiatement du fait que $G$ est un groupe diédral.  
\end{proof}

\begin{theorem}\label{dimension}On a:

\begin{enumerate}

\item $\dim \rH^1(M,\psi_{_{\heartsuit}}^{-1})=[F:\Q]$. 

\item  $\ker \left (\rH^1(M,\psi_{_{\heartsuit}}^{-1}) \rightarrow \underset{\gp_i \mid p}{\oplus} \rH^1(M_{v_i},\psi_{_{\heartsuit}}^{-1})) \right)$ est trivial.

\end{enumerate}

\end{theorem}
\begin{proof}\ 

(i) Nos hypothèses impliquent que $\Gal(H/K)$ est cyclique et que $\Gal(H/M)$ est un sous-groupe de $\Gal(H/K)$ d'indice égal à $[F:\Q]$  (où $[F:\Q]=2$), donc il existe $2$ caractères $(\psi_i)_{1\leq i \leq 2}: G_K \rightarrow \bar{\Q}_p$ qui prolongent  $\psi_{_{\heartsuit}}^{-1}$ à $\Gal(H/K)$. De plus, le lemme \ref{multiplicity} implique que $(\Ind^{\Q}_{K} \psi_i)_{\{1\leq i\leq 2\}}$ sont les représentations de $\Gal(H/\Q)$ qui prolonge la $\Gal(H/F)$-représentation $\Ind^{F}_{M}\psi_{_{\heartsuit}}^{-1}$ à $\Gal(H/\Q)$. 

Puisque $\det \Ind^{\Q}_{K} \psi_i(c)=-1$, le lemme \ref{multiplicity} implique que la représentation irréductible $\Ind^{K}_{\Q} \psi_i$ apparaît avec une multiplicité égale à $1$ dans la décomposition de $G$-modules (\ref{mutiplicity-}). Par conséquent, la multiplicité de $\Ind^{F}_{M}\psi_{_{\heartsuit}}^{-1}$ dans la $\Gal(H/F)$-représentation $\rH^{1}(G_H, \bar{\Q}_p)$ est égale à $2$, et puisque $\pi'=\Ind^{F}_M(\psi_{_{\heartsuit}}^{-1})$ est irréductible, alors
$$\dim \rH^1(H, \bar{\Q}_p)[\psi_{_{\heartsuit}}^{-1}]= \dim \Hom_{\Gal(H/F)}\left(\pi', \rH^1(H, \bar{\Q}_p)\right)=2$$

Ainsi, le résultat désiré découle de l'isomorphisme (\ref{iso-inf}).

(ii) D'après le lemme \ref{multiplicity} $$\dim \Hom_{G}(\Ind^{\Q}_{K} \psi_i, \rH^1(G_{H}, \bar{\Q}_p))=\dim \rH^{1}(G_H, \bar{\Q}_p)[\psi_i]=1 \text{ pour $1\leq i \leq 2$}$$ 
Maintenant, pour tout  $1\leq i \leq 2$, choisissons un générateur $f_i$ de $\rH^1(G_{H}, \bar{\Q}_p)[\psi_i]$. On peut voir que les $1$-cocycles $(f_i)_{1\leq i \leq 2}$ sont des éléments du groupe de cohomologie $\rH^1(H, \bar{\Q}_p)[\psi_{_{\heartsuit}}^{-1}]$ (puisque $\psi_i$ prolonge $\psi_{_{\heartsuit}}^{-1}$ à $\Gal(H/K))$. 

D'autre part, puisque $f_i \in \Hom ((\cO_{H} \otimes \Z_p)^\times, \bar{\Q}_p)$, alors $f_i$ est de la forme suivante: $$u \otimes 1 \mapsto  \sum_{g \in \Gal(H/\Q)} a^{i}_g \log_p(g^{-1}(u)).$$

Or $f_i$ est un élément du $\psi_i$-espace propre $\Hom ((\cO_{H} \otimes \Z_p)^\times, \bar{\Q}_p)$, alors:

$$a^{i}_{g} = \psi^{-1}_i(g) a^{i}_{1}   \text{   et   } a^{i}_{c g}= \psi^{-1}_i(g) a^{i}_{c} \text{  pour tout $g \in \Gal(H/K)$}  $$

Soit $w_0$ (resp. $v$) la place première de $H$ (resp. $K$) au dessus de $p$ induite par $\iota_p$. Si l'image de $f_i$ dans $\Hom(\cO_{H_{w_0}}^\times, \bar \Q_p)$ est triviale, par la preuve de Minkowski du théorème des unités de Dirichlet, on sait que tout élément de $\Hom(\prod_{{w} |v^{\sigma}} \cO_{H,{w}}^\times, \bar \Q_p)$ qui est trivial sur les unités globales doit se factoriser à travers la norme, donc $f_i$ n'appartient pas au $\psi_i$-espace propre. Par conséquent l'espace vectoriel 

$$\ker(\rH^1(G_{H}, \bar{\Q}_p)[\psi_i] \rightarrow \rH^1(G_{H_{w_{0}}}, \bar{\Q}_p))$$ est trivial et de là en déduit que $a^{i}_1$ et $a^{i}_c$ sont non nuls, puisque $\sigma$ échange les groupes de décomposition $G_{H_{w_0}}$ et $G_{H_{\sigma(w_0)}}$ ($\sigma=c$). 

Si $t_0$ est un générateur du groupe cyclique $\Gal(H/K)$ et $n'$ le cardinal de $\Gal(H/K)$, alors pour tout $1\leq i \leq 2$ on a
$$f_i =a^i_1 \sum_{0 \leq j \leq n'-1} \psi^{-1}_i(t_o^j) \log_p(t_0^{-j}(.)) + a^i_c \sum_{0 \leq j \leq n'-1} \psi^{-1}_i(t_o^j) \log_p(t_0^{-j}c(.)).$$ 

Par conséquent, en écrivant les $1$-cocycles $\{f_i\}_{1\leq i \leq 2}$ dans la base $\{\log_p(g^{-1}(.))_{g \in G}\}$, on peut extraire une sous-matrice de Vandermonde $A=(\psi^{-1}_i(t_0^{j-1}))_{\{1\leq i,j\leq 2\}}$. Puisque $\psi_i \ne \psi_j$ pour $ i\ne j$, alors $\det A \ne 0$. Donc, les $1$-cocycles $\{f_i\}_{\{1\leq i \leq 2\}}$ forment une base de $\rH^1(H, \bar{\Q}_p)[\psi_{_{\heartsuit}}^{-1}]$ et puisque $a^{i}_1$ et $a^{i}_c$ sont non nuls, on ne peut pas trouver une combinaison linéaire des $f_i$ ($1\leq i \leq 2$) qui est triviale sur tous les groupes de décomposition de $G_H$ aux places $\{w_i\}_{\gp_i \in S_p}$, donc le $\bar{\Q}_p$-espace vectoriel  $$\ker \left (\Hom(G_H,\bar{\Q}_p)[\psi_{_{\heartsuit}}^{-1}] \rightarrow \underset{\gp_i \mid p}{\oplus} \Hom(G_{H_{w_i}},\bar{\Q}_p) \right)$$
est trivial. Finalement, on peut conclure grâce au lemme \ref{Rest-Inflation}.

\end{proof}

\begin{cor}\label{CM tg} Sous les hypothèses du Théorème \ref{CM-case}, on a:

\begin{enumerate}

\item L'espace vectoriel $t^{0}_{\cD^{ord}}$ est trivial.

\item La dimension de l'espace tangent $t_{\cD}$ est $[F:\Q]+1$.

\end{enumerate}

\end{cor}

\begin{proof}\ 

(i) Puisque $S^p$ est vide, alors l'inclusion (\ref{CM relative}) devient un isomorphisme car $a=d^{\sigma}=-d \in \Hom(G_M,\bar{\Q}_p)$ et donc c'est un morphisme non ramifié en les premiers au dessus de $p$, ainsi $a=0$. Donc l'assertion découle immédiatement de la relation (\ref{CM relative}) et du Théorème \ref{dimension}.

(ii) D'après l'assertion (i), il existe un isomorphisme $t_{\cD} \simeq  \Hom(G_M, \bar{\Q}_p)$. D'autre part, puisque $M$ est une extension totalement complexe de $\Q $ de degré $2[F:\Q]$ et que la conjecture de Leopoldt est vraie pour $H$, $M$ possède trois $\Z_p$-extensions indépendantes et $\dim \Hom(G_M, \bar{\Q}_p)=[F:\Q]+1$.

\end{proof}

\section{Preuves des isomorphismes $\cR \simeq \cT$}\label{main}\

On va démontrer dans cette section les théorèmes \ref{main-thm-RM}, \ref{CM-case} et \ref{thm M mixt}. 

La variété rigide de Hecke-Hilbert $\cE$ est réduite et équidimensionelle de dimension $[F:\Q]+1$, donc $\cE$ est lisse en $x$ si, et seulement si, l'espace tangent de $\cT$ est de dimension $[F:\Q]+1$. De plus, le morphisme $\kappa$ est étale en $x$ si, et seulement si, l'anneau local $\cT'$ de la fibre de $\kappa(x)$ en $x$ est isomorphe à $\bar{\Q}_p$.

\subsection{Le lieu quasi-ordinaire de la variété de Hecke-Hilbert $\cE$}\

Dans la proposition suivante, on montre qu'on a une déformation quasi-ordinaire de $\rho$ en $p$ à valeurs dans $\cT$.

\begin{prop} \label{Deformation to T}\
\begin{enumerate}

\item Il existe une représentation continue $$\rho_{\cT} : G_{F,\mathfrak{n} p} \rightarrow \GL_2(\cT),$$  telle que $\Tr \rho_{\cT}(\Frob_\ell) = T_\ell$ pout tout $\ell \nmid \mathfrak{n} p$. 
La réduction de $\rho_{\cT}$ modulo $\gm_{\cT}$  est isomorphe à $\rho$. 

\item La déformation $\rho_{\cT}$ est quasi-ordinaire en $p$.

\end{enumerate}
\end{prop}

\begin{proof}\

(i)Même preuve que la proposition \cite[6.1]{D-B}.

(ii)D'après \cite[6.3.6]{chenevier-thesis}, il existe un ouvert affinoïde de $\Omega$ de $\cE$ contenant $x$ tel que l'application $z \rightarrow |U_p(z)|_p$ est constante et égale à $1$ sur $\Omega(\C_p)$, où $w(\Omega)$ est un ouvert affinoïde de $\cW_F$, $\kappa: \Omega \rightarrow w(\Omega)$ est fini et de restriction surjective sur chaque composante irréductible de $\Omega$. Puisque le lieu quasi-ordinaire  $\cE^{n.ord}$ de $\cE$ est la fibre générique de l'algèbre de Hecke quasi-ordinaire en $p$ induite par les opérateurs de Hecke $T_{\ell}$ et $U_{\gp_i}$ où $\gp_i \mid p$ et $\ell \nmid p\gn$ et que l'anneau $\cT$ est equidimensionnel, alors les idéaux minimaux de $\cT$ correspondent aux familles de Hida quasi-ordinaires en $p$.

Soient $V_{\cT}$ le $\cT$-module de rang $2$ sur lequel $G_F$ agit par $\rho_{\cT}$ et $(K_i)_{1\leq i \leq r}$ la famille des corps obtenus par la localisation de $\cT$ en ses idéaux minimaux et posons $V_i=V\otimes_{\cT}K_i$ ($1\leq i \leq r$). $V_i$ est un $K_i[G_{F,\gn p}]$-module et est la représentation galoisienne d'une famille de Hida quasi-ordinaire en $p$ qui se spécialise en $f$. D'après un résultat de \cite{hida89}, pour tout $\mathfrak{p}_k \mid p$
il existe une suite exacte courte de $K_i[G_{F_{\gp_{k}}}]$-modules: 
\begin{equation*}
0\rightarrow  V_i^+\rightarrow  V_i \rightarrow  V_i^-\rightarrow 0,
\end{equation*}
où $V_i^-$ est une droite sur laquelle $G_{F_{\gp_{k}}}$ agit par le caractère $\delta_{\gp_{k}}$ qui envoie $[y, F_{\gp_{k}}]$ sur l'opérateur de Hecke $T(y)$, où $[.,F_{\gp_{k}}]:\widehat{F_{\gp_{k}}^{\times}} \rightarrow G^{ab}_{F_{\gp_{k}}}$ est le symbol local d'Artin. Puisque $\rho$ est $p$-régulière et $\delta_{\gp_{k}}$ relève le caractère $\psi''_k$, on peut conclure par le même argument que celui utilisé dans la proposition \cite[6.1]{D-B} que $\rho_\cT$ est quasi-ordinaire en $p$.

\end{proof}

\subsection{Surjection de $\cR$ sur $\cT$}\

D'après la proposition \ref{Deformation to T}, la représentation $\rho_{\cT}$ induit un morphisme local continu de $\bar \Q_p$-algèbres
\begin{equation}\label{nearly-ord univ}
\cR  \rightarrow \cT.
\end{equation}

Soit $A$ un anneau local Artinien d'idéal maximal $\gm_A$ et de corps résiduel $A/\gm_A=\bar \Q_p$. Toute déformation de $\det(\rho)$ (resp. $(\det \rho,(\psi_{i}'')_{\gp_i \mid p}))$ à valeurs dans $A^\times$ est équivalente à un morphisme continu de $G_{F, \mathfrak{n}p}$  (resp. $G_{F, \mathfrak{n}p} \times (\mathfrak{o} \otimes \Z_p)^{\times}$) vers $1+\gm_A$. 

D'autre part, la restriction de la déformation universelle $\rho_{\cR}$ à $G_{F_{\gp_i}}$ pour tout $\gp_i \mid p$ est une extension d'un caractère $\psi_{i,\cR}''$ qui relève $\psi_{i}''$ par un caractère $\psi_{i,\cR}'$ qui relève $\psi_{i}'$. D'après la théorie du corps de classes et puisque $1+\gm_A$ ne contient pas d'élément d'ordre fini, l'anneau universel qui représente les déformations de $\det \rho$ (resp. $\det \rho \times (\psi_{i}'')_{\{\gp_i \mid p \}}$) est donné par $\bar\Q_p[[1+p\Z_p]]$ (resp. $\bar\Q_p[[T_1,T_2, ...., T_{n+1}]]$). De plus, $\cR^{ord} $ (resp. $\cR $) a une structure de $\bar \Q_p[[1+p\Z_p]] $-algèbre (resp. $\bar\Q_p[[T_1,T_2, ...., T_{n+1}]]$-algèbre) donnée par la déformation $\det \rho_{\cR^{ord}}$ (resp. $(\det \rho_{\cR}, (\psi_{i,\cR}'')_{\{\gp_i \mid p \}})$) de $\det \rho$ (resp. $(\det \rho,(\psi_{i}'')_{\gp_i \mid p})$).
 
L'action de $\Lambda$ sur l'anneau $\cT$ donnée par le morphisme poids $\kappa$ est compatible avec l'action induite par les caractères $\det \rho_{\cT}$ et ceux donnés par l'image de $\psi_{i,\cR}''$ via le morphisme (\ref{nearly-ord univ}) (voir \cite{hida89}). Ainsi, le morphisme (\ref{nearly-ord univ}) est $\Lambda$-linéaire et le diagramme suivant est commutatif:

 $$\xymatrix{  \bar \Q_p[[T_1,T_2,...,T_{n+1}]]    \ar@{->>}[d] \ar@{->}[r] & \cR \ar@{->}[d] \\
\Lambda  \ar@{->}[r] &\cT.}$$

\begin{lemma} \label{Compatibility } Le morphisme naturel $\bar \Q_p[[T_1,T_2,...,T_{n+1}]] \rightarrow \Lambda$ est un isomorphisme.

\end{lemma}
\begin{proof}\ 

Puisque l'anneau $\Lambda$ est équidimensionnel de dimension $n+1$, la surjection $$\bar \Q_p[[T_1,T_2,...,T_{n+1}]]\rightarrow \Lambda$$ est un isomorphisme d'anneaux réguliers. 
  \end{proof}

\begin{prop}\label{prop-surj}
Le morphisme (\ref{nearly-ord univ}) est surjectif. 
\end{prop}
\begin{proof}\

Puisque $\cT$ est topologiquement engendré sur $\Lambda$ par $U_{\gp}$ pour tout $\gp \mid p$ et  $T_\ell$ pour  $\ell \nmid \gn p$, il suffit de montrer que ces éléments sont dans l'image de $\cR$ via le morphisme (\ref{nearly-ord univ}).
 
Pour $\ell \nmid  \gn p$,  $T_\ell=\Tr \rho_{\cT}(\Frob_\ell)$ est l'image de la trace de $\rho_{\cR}(\Frob_\ell)$. 

D'autre part, la restriction de $\rho_{\cR}$ à $G_{F_{\gp_i}}$ pour tout $\gp_i \mid p$ est une extension du $\psi_{i,\cR}''$ par le caractère $\psi_{i,\cR}'$, où l'image du caractère $\psi_{i,\cR''}$ dans $\cT$ est le caractère $\delta_{\gp_i}$ qui envoie $[y, F_{\gp_{i}}]$ sur l'opérateur de Hecke $T(y)$ où $[.,F_{\gp_i}]:\widehat{F_{\gp_i}^{\times}} \rightarrow G^{ab}_{F_{\gp_i}}$ est le symbol d'Artin. Ainsi, $U_{\mathfrak{p}_i}=[ \pi_{\gp_i}, F_{\gp_i}]$ est dans l'image du morphisme (\ref{nearly-ord univ}) pour une certaine uniformisante $\pi_{\gp_i}$ du corps complet $F_{\gp_i}$.

\end{proof}

D'après la proposition \ref{Compatibility }, $\varLambda$ représente le problème de déformation de $(\det \rho,(\psi_{i}'')_{\gp_i \mid p})$ aux éléments inversibles d'une $\bar{\Q}_p$-algèbre artinienne $A$. Donc, le foncteur oubli des déformations de $(\det \rho,(\psi_{i}'')_{\gp_i \mid p})$ aux déformations $\det \rho$ induit la surjection suivante: 

$$\varLambda\twoheadrightarrow \bar{\Q}_p[[1+p\Z_p]].$$

Le noyau $\mathcal{I}^{aug}$ de la surjection ci-dessus est appelé l'idéal d'augmentation.

On rappelle que $$\cT^{ord}=\cT/\mathcal{I}^{aug} \cT$$ 

En utilisant le même argument, on obtient l'isomorphisme suivant: 
\begin{equation}\label{augmentation}
\cR/\mathcal{I}^{aug} \cR \simeq \cR^{ord}
\end{equation}

On note $\cR'$ le quotient de $\cR$ par l'idéal engendré par l'image de l'idéal maximal de $\varLambda$.

Il est facile de voir que $\cR'$ est le plus grand quotient $p$-ordinaire de $\cR$ dans lequel $\rho$ peut être déformée avec un déterminant constant. Ainsi, l'anneau local $\cR'$ représente $\cD^{ord}_{\det \rho}$.

\subsection{Démonstration du Théorème \ref{main-thm-RM} et du Théorème \ref{CM-case}.}\

\begin{theorem}\label{n,ord iso}Sous les hypothèses du Théorème \ref{main-thm-RM} ou du Théorème \ref{CM-case}, le morphisme (\ref{nearly-ord univ}) est un isomorphisme d'anneaux réguliers.

\end{theorem}
\begin{proof}\

Le corollaire \ref{CM tg} et le théorème \ref{dim c} impliquent que $\dim t_{\cD}=[F:\Q]+1$, donc la dimension de l'espace tangent de $\cR$ est $[F:\Q]+1$. D'autre part,  l'anneau local $\cT$ est equidimensionnel de dimension $[F:\Q]+1$ (puisque $\cE$ est equidimensionnelle de dimension $[F:\Q]+1$) et par consequent, la proposition \ref{prop-surj} implique que (\ref{nearly-ord univ}) est un isomorphisme d'anneaux locaux réguliers de dimension $[F:\Q]+1$. 
 
 \end{proof}
 
L'isomorphisme d'anneaux locaux $\cR \simeq \cT$ et (\ref{augmentation}) impliquent qu'il existe des isomorphismes $\cR^{ord} \simeq \cT^{ord}$ et $\cR' \simeq \cT'$. Ainsi, on peut conclure grâce au théorème \ref{dim c} dans le cas où $M$ est totalement réel. Dans le cas où $M$ est totalement complexe, le corollaire \ref{CM tg} implique que $\cT' \simeq \bar{\Q}_p$ et donc $\kappa$ est étale en $x$, ce qui achève notre démonstration.

\subsection{$\cR_\m=\cT$ dans la cas où l'extension $M/F$ n'est pas totalement réel.}\

Soient $h^{n,ord}(\mathfrak{n},\cO)$ l'algèbre de Hecke quasi-ordinaire en $p$ construite par Hida dans \cite{hida90} et $h'$ { \it la sous-algèbre de Hecke} de $h^{n,ord}(\gn,\cO)$ engendrée sur l'algèbre d'Iwasawa par les opérateurs de Hecke $\{U_\gp,T_{\ell},\textless \ell \textgreater \}_{\gp \mid p,\ell \nmid \gn p}$ (On a omit les operateurs $U_\mathfrak{q} $ pour $ \mathfrak{q} \mid \gn$). Il est connu que le lieu quasi-ordinaire de $ \cE^{n.ord}$ de $\cE$ est la fibre générique de l'algèbre de Hecke $h'$. Puisque $f$ est ordinaire en $p$, il existe un unique morphisme $\varphi^{n,ord}_{\theta}:h' \rightarrow \cO$ qui envoie $T_\ell$ sur $\Tr \rho(\Frob_\ell)$ pour tout $\ell \nmid \gn p$.

On note $\mathcal{P}^{n,ord}_{\theta}$ l'idéal premier $\ker \varphi^{n,ord}_{\theta}$ et $\cH$ le complété de l'anneau local de $\Spec h'$ en $\mathcal{P}^{n,ord}_{\theta}$. Après une extension par des scalaires, on peut supposer que $\cH$ contient $\bar{\Q}_p$.

Soit $\mathfrak{C}_{\cO}$ la catégorie des $\cO$-algèbres locales complètes noethériennes de corps résiduel $\mathbb{F}$ et dont les morphismes sont les homomorphismes locaux d'anneaux locaux qui induisent l'identité sur les corps résiduels. 

On a les conditions suivantes:

\begin{itemize}

\item[$\bullet$] $({ \bf AI}_F)$  { \small La restriction $\bar{\rho}$ à $G_{F(\sqrt{(-1)^{(p-1)/2}p)}}$ est absolument irréductible.}

\item[$\bullet$] $({ \bf LD}_p)$ { \small  $F$ est linéairement disjointe de $\Q(\zeta_p)$ sur $\Q$ et $\mu_p(F_{\gp})=\{1\}$ pour tout $\gp \mid p$, où $\zeta_p$ est une racine $p$-ième de l'unité et $\mu_p(F_\gp)$ est l'ensemble des racines $p$-ième de l'unité de $F_{\gp}$.}
\end{itemize}

La condition $({ \bf LD}_p)$ est plus faible que la condition: $p$ est non ramifié dans $F$.

On suppose dans cette sous-section que la représentation résiduelle $\rho \otimes \mathbb{F}=\bar{\rho}=\Ind^{F}_{M} \bar{\psi} $ est absolument irréductible, minimale au sens de Fujiwara \cite{Fujiwara}, $p$-distinguée et l'ordre de $\psi$ est premier avec $p$. D'après les critères de Schlesinger, le foncteur $\cD_{\bar{\rho}}$ des déformations de $\bar{\rho}$ qui sont minimales et non ramifiées en dehors de $\gn p$ et quasi-ordinaires en $p$ à valeurs dans les objets de la catégorie $\mathfrak{C}_\cO$ est représentable par le couple $(\cR_{\bar{\rho}}^{\m},\rho^{n,ord}_{\mathbb{F}})$.

D'autre part, on note $h^{n,ord}$ la composante locale $h^{n,ord}(\gn,\cO)$ correspondant à $\bar{\rho}$ ($h^{n,ord}(\gn,\cO)$ est semi-local) et $\rho_{h^{n,ord}}:G_{F,S}\rightarrow \GL_2(h^{n,ord})$ la déformation quasi-ordinaire en $p$ de $\bar{\rho}$ construite par Hida dans \cite{hida90} qui envoie $\Frob_\ell$ vers $T_\ell$ pour $\ell \nmid \gn p$. Sous les hypothèses  $({ \bf AI}_F)$ et $({ \bf LD}_p)$, le théorème de Fujiwara's \cite{Fujiwara} appliqué à la représentation résiduelle $\bar{\rho}$ nous donne que $\cR_{\bar{\rho}}^{\m} \simeq h^{n,ord}$. 

Pour tout premier $\mathfrak{q}$ de $F$ divisant $\gn$, on note $a_{\mathfrak{q}}$ la valeur propre de $f$ pour l'opérateur de Hecke $U_{\mathfrak{q}}$. Soit $\cD_{\m}$ le sous-foncteur de $\cD$ qui consiste en les déformations $\rho_A$ qui sont minimales, dans le sens où pour tout $\mathfrak{q} \mid \gn$ tel que $a_{\mathfrak{q}}\ne 0$, l'espace des $I_{\mathfrak{q}}$-invariants dans $\rho_A$ est un $A$-module libre de rang $1$. Le foncteur $\cD_{\m}$ est représentable par $(\cR_\m,\rho_{\cR_\m})$. Puisque la représentation $\rho$ est une déformation de $\bar{\rho}$, donc $\cD_{\m}$ est la fibre générique du foncteur $\cD_{\bar{\rho}}$ (voir \cite[2.3.5]{Kisin}). Ainsi, il existe un isomorphisme (après une extension par des scalaires) entre $\cR_{\m}$ et la complétion de la localisation de $\cR_{\bar{\rho}}^{\m}$ par l'idéal premier de hauteur $1$ donné par le noyau du morphisme $\cR_{\bar{\rho}}^{\m} \rightarrow \cO$ qui induit la déformation $\rho$ de $\bar{\rho}$. De plus, la déformation $\rho_{\cR_\m}$ est le tiré-en-avant de $\rho^{n,ord}_{\mathbb{F}}$ par le morphisme de localisation $\cR_{\bar{\rho}}^{\m} \rightarrow \cR_\m$. 

\begin{lemma}\label{R=T M mixt}Avec les notations ci-dessus, on a:

\begin{enumerate}

\item Il existe un isomorphisme $(\cH, \rho_{\cH})\simeq (\cR_\m,\rho_{\cR_\m})$. 

\item Il existe un isomorphisme $\cH \simeq \cT$.

\item L'application canonique $\Spec h^{n,ord}(\gn,\cO) \rightarrow \Spec h'$ induit un isomorphisme sur les complétés des anneaux locaux aux points associés à $f$.

\end{enumerate}

\end{lemma}

\begin{proof}\
(i) D'après Nyssen \cite{nyssen} et Rouquier \cite{Rouquier}, $\cR_\m$ est engendré sur $\varLambda$ par la trace de $\rho_{\cR_{\m}}$. Donc, l'isomorphisme $(\cH, \rho_{\cH})\simeq (\cR_\m,\rho_{\cR_\m})$ découle immédiatement de l'isomorphisme $\cR_{\bar{\rho}}^{\m}\simeq h^{n,ord}$, du fait que le foncteur $\cD_{\m}$ est la fibre générique du foncteur $\cD_{\bar{\rho}}$ et du lemme \cite[2.3.5]{Kisin}. 

(ii) D'après la construction de la variété $p$-adique rigide $\cE$, il existe un isomorphisme de variétés analytiques rigides $\cE^{n.ord}\simeq (\Spec h')^{rig}$ induisant un isomorphisme naturel sur les complétés des anneaux locaux en tout point. 

(iii) Puisque $\bar{\rho}$ est minimale, $U_{\mathfrak{q}} \in \cH$ pour tout $\mathfrak{q} \mid \gn$. 
\end{proof}

\subsection{Preuve du Théorèmes \ref{thm M mixt}.}\

Soit $\cD_{\m}^{ord}$ le sous-foncteur de $\cD_{\m}$ qui consiste en les déformations $\rho_A$ qui sont ordinaires en $p$. Le foncteur $\cD_{\m}^{ord}$ est représentable par $\cR_{\m}^{ord}$.  

On a l'isomorphisme suivant: 
\begin{equation}\label{augmentation-M ni r ni c}
\cR_{\m}/\mathcal{I}^{aug} \cR_{\m} \simeq \cR^{ord}_{\m}
\end{equation}

On note $\cR'_{\m}$ pour le quotient de $\cR_{\m}$ par l'idéal engendré par l'image de l'idéal maximal de $\varLambda$.

On remarque que $\cR'_{\m}$ est le plus grand quotient $p$-ordinaire de $\cR_{\m}$ dans lequel $\rho$ peut être déformée avec un déterminant constant. Ainsi, l'anneau local $\cR'_{\m}$ représente le sous-foncteur $\cD^{ord}_{\det \rho}$ qui consiste en les déformations qui sont minimales.

D'après le lemme \ref{R=T M mixt}, on a un isomorphisme entre $\cR_{\m} \simeq \cT$. De plus, (\ref{augmentation-M ni r ni c}) impliquent qu'il existe des isomorphismes entre $\cR^{ord}_{\m} \simeq \cT^{ord}$ et $\cR'_{\m} \simeq \cT'$. Dans nos calculs de l'espace tangent du foncteur $\cD$, la condition minimale en $\mathfrak{q}\mid  \gn$ est toujours vraie pour les $1$-cocycles, puisque n'importe quel élément de $\Hom(G_H,\bar{\Q}_p)$ est non ramifié en dehors de $p$ et $\rho$ est d'image finie. Ainsi, les espaces tangents de $\cR'$ et de $\cR'_{\m}$ sont égaux et donc l'espace tangent de $\cR'_\m$ est non trivial d'après la proposition \ref{tg M mixt}, ce qui achève notre démonstration.

\subsection{La courbe de Hecke-Hilbert parallèle $\cC_F$ au point $x$}\

La courbe de Hecke-Hilbert parallèle $\cC_F$ est équidimensionelle de dimension $1$.

\begin{cor}On a:

\begin{enumerate}

\item Sous les hypothèses du théorème \ref{thm M mixt} et si $M$ a au plus $2([F:\Q]-2)$ plongements complexes et $S^p$ est vide, alors $f$ est un point singulier de la courbe de Hecke-Hilbert $\cC_F$.

\item Supposons que $M$ est totalement réel, la conjecture de Leopoldt est vraie pour $M$ et $\# S_p \geq 2$, alors $f$ est un point singulier de la courbe de Hecke-Hilbert $\cC_F$.

\end{enumerate}
\end{cor}

\begin{proof}\

(i) Puisque $\cR^{ord}_{\m}\simeq \cT^{ord}$, la proposition \ref{tg M mixt} implique que la dimension de l'espace tangent de $\cT^{ord}$ est au moins $2$, or la dimension de Krull de $\cT^{ord}$ est $1$ (puisque $\cC_F$ est equidimensionnelle de dimension $1$), d'où le résultat.

(ii) Même argument que dans (i) combiné avec le théorème \ref{dim c}.

\end{proof}

\begin{rem}
La variété de Hecke existe si $N_\Q^F(\mathfrak{n}) \geq 4$, car sa construction dépend de l'existence d'un schéma qui représente l'espace de modules de Hilbert-Blumenthal, et pour $N_\Q^F(\mathfrak{n}) < 4$, cet espace de modules n'est pas representable par un schéma. 

Puisque l'algèbre de Hecke $h'$ est finie et sans torsion sur l'algèbre d'Iwasawa, alors les théorèmes \cite[2.4]{hida88} et \cite[3]{wiles} impliquent que $\cH$ est equidimensionnel de dimension $[F:\Q]+1$, et ainsi si $N_\Q^F(\mathfrak{n}) < 4$, nos théorèmes restent vraie après avoir remplacer $\cT$ par $\cH$.

\end{rem}

\section{Preuve du Th\'eor\`eme $0.5$}

On démontre le Théorème \ref{q exp} dans cette section en utilisant les techniques de \cite{darmon}. L'existence du sous-groupe canonique sur un voisinage strict du lieu ordinaire de la variété de Hilbert induit l'injection suivante $$ \jmath: S_{1}(\gn p, \chi)[f]  \hookrightarrow S^{\dag}_{1}(\gn, \chi)[[f]],$$où $S^{\dag}_{1}(\gn, \chi)[[f]]$ est l'espace propre généralisé associé à $f$ dans l'espace des formes surconvergentes de poids $1$. En utilisant le même argument que celui déjà utilisé dans la démonstration de la proposition \cite[1.2]{darmon}, l'inclusion $\jmath$ est un isomorphisme si, et seulement si, $\cE$ est étale sur $\cW_F $ en $x$. 

Sous les hypothèses du Théorème \ref{main-thm-RM} et si de plus $S_{p}$ est non vide, alors il existe un morphisme surjectif $\pi :\cT' \rightarrow \bar{\Q}_p[\epsilon]$ de noyau $I_\pi$. $I_\pi$ annule un sous-espace $S^{\dag}_{1}(\gn, \chi)[I_\pi]$ de dimension $2$ de $S^{\dag}_{1}(\gn, \chi)[[f]]=\Hom_{\bar{\Q}_p}(\cT',\bar{\Q}_p)$ et ce sous-espace contient une forme propre généralisée normalisée $f^{\dag}$.

Pour tout idéal premier $\mathfrak{q}$ de $F$, on note respectivement $a_\mathfrak{q}(f^{\dag})$ et $a_\mathfrak{q}(f)$ les coefficients de Fourier de $f^\dag$ et $f$. D'autre part, avec le même argument utilisé dans \cite{darmon}, l'opérateur de Hecke $T_\gq$, où $\gq \nmid \gn p$, agit sur $f^{\dag}$ de la manière suivante
 \begin{equation}
 \label{eqn:hecke} { \small 
 T_\mathfrak{q} f^{\dag} = a_\mathfrak{q}(f) f^{\dag} + a_\mathfrak{q}(f^{\dag} ) f  }
 \end{equation}
  
On rappelle que $H$ est le corps fixé par $\ker \psim$ et que $G''=\Gal(H/M)$. Soit $\ell \nmid \gn p$ un premier de $F$ qui est inerte dans $M$, alors le premier de $M$ au dessus de $\ell$ se décompose complètement dans $H$ (puisque $\rho$ est non ramifié en $\ell$). On note $\Sigma_\ell$ l'ensemble des premiers de $H$ au dessus de $\ell$. Puisque l'extension $H/M$ est galoisienne, $G''$ agit sur $\Sigma_\ell$. Soient $\lambda \in \Sigma_\ell$ et $u_\lambda \in \cO_H[1/\lambda]^\times \otimes \Q$ une $\lambda$-unité de $H$ de $\lambda$-valuation égale à $1$. $u_\lambda$ est défini à une unité près et l'élément 

 \begin{equation}\label{unit lambda} 
u(\psim,\lambda, g_i) = \sum_{h \in G''} \psi_{_{\heartsuit}}(h) \otimes g_i \circ h(u_\lambda) \in \bar{\Q} \otimes   g_i (\cO_H)[1/\ell]^\times \text{ où $g_i \in I'_F$}
\end{equation}
est indépendant du choix de $u_\lambda$, puisque $H$ est CM (car $M$ est totalement réel) et donc on n'a pas d'unité de $\bar{\Q}\otimes \cO_H^\times$ dans le 
 $\psim$-sous-espace propre.
 
Soit $e_2$ un vecteur propre de $\bar{\Q}_p^2$ pour $\psi^{\sigma}$ et $e_1 = \rho(\sigma) e_2$.  Dans la base $(e_1,e_2)$, on a
\begin{equation}\label{rho}
\rho_{|G_M} = \left(\begin{matrix} \psi & 0 \\ 0 & \psi^{\sigma} \end{matrix}\right) \qquad \rho_{|G_F \setminus G_M} = \left(\begin{matrix} 0 & \eta' \\ \eta & 0 \end{matrix}\right)
\end{equation}
où $\eta$ et $\eta'$ sont des fonctions de $G_F \setminus G_M$ à valeurs dans $\bar{\Q}$ définies par $\eta(h)=\psi(\sigma h)$ et $\eta'(h):=\psi(h \sigma^{-1})$.

Soient $\sigma_\lambda \in \Gal(H_\lambda/F_\ell)\subset \Gal (H/F)$ le Frobenius en $\ell$ attaché à la place première $\lambda$ et $H_{\psi}$ le corps de nombres fixé par $\psi$ ($H\subset H_{\psi}$). D'après la relation \cite[(13)]{darmon}, $\eta(\sigma_\lambda)$ ne dépend pas du choix du premier de $H_{\psi}$ au dessus de $\lambda$. 

On a les relations suivantes (voir \cite[(13)]{darmon}) 
\begin{equation}\label{N darmon}
\begin{split}
\forall & g \in \Gal(H/M) \text{, } \eta(\sigma_{g(\lambda)})=\psi(\sigma \sigma_{g(\lambda)})=\psi(\sigma g^{-1}\sigma_{\lambda}g)=\psim(g)\psi(\sigma \sigma_{\lambda})=\psim(g) \eta(\sigma_\lambda) \\
\forall & g \in \Gal(H/M) \text{, } u(\psim,g(\lambda), g_i)= \psi_{_{\heartsuit}}^{-1}(g) u(\psim,\lambda, g_i)
\end{split}
\end{equation}

D'après la relation (\ref{N darmon}), l'élément
\begin{equation}
\label{eqn:def-u-l}
 u(\psim,\ell, g_i) =  \psi(\sigma \sigma_{\lambda}) \otimes u(\psim,\lambda, g_i)  \in \bar{\Q} \otimes g_i(\cO_H)[1/\ell]^\times 
 \end{equation}
 dépend uniquement de $\ell$ et non du choix du premier $\lambda\in \Sigma_\ell$.

D'après le (ii) du théorème \ref{dim c}, le morphisme $\Psi': (\cO_H \otimes \Z_p)^{\times} \rightarrow \bar{\Q}_p$ défini par $$\left(u\otimes 1 \mapsto \sum_{g_i \in I'_F} \alpha_i  \sum_{g \in G''} \psim(g) \log_p \big(\iota_p\circ g_i \circ g(u) \big) \right)$$ 

est associé à un unique élément de l'espace tangent de $\cT'$ (i.e $\Psi'=b_{|G_H}=c^{\sigma}_{|G_H} \in \rH^{1}(H,\bar{\Q}_p)[\psim]$) et donne aussi un unique morphisme surjectif $\pi:\cT' \rightarrow \bar{\Q}_p[\epsilon]$ associé à $\Psi'$ qui induit une déformation $\tilde{\rho}:G_F \rightarrow \GL_2(\bar{\Q}_p[\epsilon])$ de déterminant égal à $\det \rho$.

D'après l'isomorphisme (\ref{iso-inf}), on peut étendre $\Psi'$ à un $1$-co-cycles de $\rH^{1}(G_M,\psim)$ de manière unique (à co-bords près), donc la déformation $\tilde{\rho}$ est de la forme suivante :

\begin{equation}\label{tilde}
\tilde\rho_{|G_M} = \left(\begin{matrix} \psi & \psi^{\sigma} \Psi'.\epsilon \\ \psi \Psi. \epsilon & \psi^{\sigma} \end{matrix}\right) \qquad \rho_{|G_F \setminus G_M} = \left(\begin{matrix} d_1. \epsilon & \eta' \\ \eta & d_2.\epsilon \end{matrix}\right)
\end{equation}

H.Darmon, A.Lauder and V.Rotger ont montré dans \cite{darmon} les propositions suivantes.

\begin{lemma}\cite[2.1]{darmon}
La fonction $\Psi \left( \text{resp. } \Psi' \right)$ appartient au groupe de cohomologie $\rH^1(M,\psi_{_{\heartsuit}}^{-1})$ ($\text{resp. } \rH^1(M, \psim)$). De plus, on a les relations suivantes après restriction des fonctions ci-dessus à $G_H$:
\begin{equation}
\label{kappa' and kappa}
 \Psi'(h) = \frac{\eta'(\tau)}{\eta(\tau)} \Psi(\tau h \tau^{-1}),
 \end{equation}
où $\tau \in G_F \setminus G_M$. 
\end{lemma}

\begin{lemma}\cite[2.2]{darmon}
\label{lemma:cft}
Pour tout premier $\ell$ de $F$ inerte dans $M$ et pour tout $\lambda \in \Sigma_\ell$, on a
$$ \Psi'(\Frob_{\lambda}) = \sum_{g_i \in I'_F} \alpha_i \log_p(u(\psim,\lambda, g_i)).$$
\end{lemma}

La déformation $\tilde\rho$ est la représentation galoisienne associée à une forme modulaire propre surconvergente de poids $1$ à coefficients dans $\bar{\Q}_p[\epsilon]$, notée $\tilde{f} =f +  f^\dag \varepsilon$. La forme modulaire $f^\dag$ est une forme propre généralisée et normalisée comme dans l'introduction et ses coefficients de Fourier peuvent être interprétés de la manière suivante:
 
  \begin{equation}
 \label{gal surconvergent}
\Tr (\tilde\rho(\Frob_\ell)) = a_\ell(f) +   a_\ell(f^\dag) \varepsilon,
 \end{equation}
  où $\ell\nmid \gn p$. 

Soit $\ell\nmid \gn p$ un premier de $F$ qui se décompose dans $M$, alors $\Frob_\ell \in G_M$ et donc
  $\Tr(\tilde\rho)(\Frob_\ell) = a_\ell(f)$. Ainsi, on a démontré la première partie du Théorème \ref{q exp}.

Maintenant, soit $\ell\nmid \gn p$ un premier de $F$ qui est inerte dans $M$. On fixe un plongement $\iota_\ell: \bar\Q \hookrightarrow \bar\Q_\ell$ qui induit un plongement du groupe de décomposition $G_{F_{\ell}}$ dans $G_F$. Si $\Frob_\ell$ est un Frobenius de $G_F$ en $\ell$ alors $\Frob_\ell^2$ est le Frobenius de $G_M$ associé au premier de $M$ au dessus de $\ell$. Soit $\lambda \in \Sigma_\ell$ la place première canonique de $H$ au dessus de $\ell$ induite par $\iota_\ell$ ($\sigma_{\lambda}=\Frob_\ell$).

D'après les relations (\ref{tilde}) et (\ref{gal surconvergent}),
 \begin{equation}
 \label{Tr-ell}
  \Tr(\tilde\rho(\Frob_\ell)) = (d_1(\Frob_\ell) + d_2(\Frob_\ell)) \epsilon =  a_\ell(f^\dag) \epsilon,
  \end{equation}
 D'autre part, en utilisant (\ref{tilde}), un calcul direct de $\tilde\rho(\Frob_\ell^2)$ implique
\begin{equation} \left(\begin{matrix} * & \psi^{\sigma}(\Frob_\ell^2) \Psi'(\Frob_\ell^2) \epsilon \\
 * &  * \end{matrix} \right)= 
 \left(\begin{matrix} *  & \eta'(\Frob_\ell) (d_1(\Frob_\ell) + d_2(\Frob_\ell)) \epsilon \\
  * & * \end{matrix} \right)
   \end{equation}
 
Donc, $ \psi^{\sigma}(\Frob_\ell^2) \Psi'(\Frob_\ell^2) = \eta'(\Frob_\ell) (d_1(\Frob_\ell) + d_2(\Frob_\ell))$.
 
Finalement, il découle de la relation (\ref{gal surconvergent}) et de la proposition \cite[2.2]{darmon} que 
 $$ a_\ell(f^\dag) = \eta(\Frob_\ell) \Psi'(\Frob_\ell^2) = \eta(\lambda) \Psi'(\Frob_\lambda). $$

\bibliographystyle{siam}

\end{document}